\documentclass[final,12pt,backref=page]{colt2025} 

\title[Faster Acceleration for Steepest Descent]{Faster Acceleration for Steepest Descent}

\usepackage{mathtools}
\newcommand{\R}{\mathbb{R}}

\newcommand{\defeq}{\coloneqq}

\newcommand{\eps}{\varepsilon}
\newcommand{\grad}{\nabla}

\DeclarePairedDelimiter{\crl}{\{}{\}}

\DeclareMathOperator*{\argmin}{arg\,min}

\newcommand{\mc}[1]{\mathcal{#1}}

\def\ddefloop#1{\ifx\ddefloop#1\else\ddef{#1}\expandafter\ddefloop\fi}
\def\ddef#1{\expandafter\def\csname 
bb#1\endcsname{\ensuremath{\mathbb{#1}}}}
\ddefloop ABCDEFGHIJKLMNOPQRSTUVWXYZ\ddefloop
\def\ddefloop#1{\ifx\ddefloop#1\else\ddef{#1}\expandafter\ddefloop\fi}
\def\ddef#1{\expandafter\def\csname 
b#1\endcsname{\ensuremath{\mathbf{#1}}}}
\ddefloop ABCDEFGHIJKLMNOPQRSTUVWXYZ\ddefloop
\def\ddef#1{\expandafter\def\csname 
c#1\endcsname{\ensuremath{\mathcal{#1}}}}
\ddefloop ABCDEFGHIJKLMNOPQRSTUVWXYZ\ddefloop
\def\ddef#1{\expandafter\def\csname 
h#1\endcsname{\ensuremath{\widehat{#1}}}}
\ddefloop ABCDEFGHIJKLMNOPQRSTUVWXYZ\ddefloop
\def\ddef#1{\expandafter\def\csname 
hc#1\endcsname{\ensuremath{\widehat{\mathcal{#1}}}}}
\ddefloop ABCDEFGHIJKLMNOPQRSTUVWXYZ\ddefloop
\def\ddef#1{\expandafter\def\csname 
t#1\endcsname{\ensuremath{\widetilde{#1}}}}
\ddefloop ABCDEFGHIJKLMNOPQRSTUVWXYZ\ddefloop
\def\ddef#1{\expandafter\def\csname 
tc#1\endcsname{\ensuremath{\widetilde{\mathcal{#1}}}}}
\ddefloop ABCDEFGHIJKLMNOPQRSTUVWXYZ\ddefloop

\usepackage{nicefrac}

\newsavebox\CBox

\newcommand{\norm}[1]{\left \| #1 \right \|}
\newcommand{\abs}[1]{\left\lvert #1\right\rvert} %

\newcommand{\pa}[1]{\left( #1 \right)}
\newcommand{\bracks}[1]{\left[ #1 \right]}
\newcommand{\inner}[2]{\left\langle #1,\, #2 \right\rangle}

\newcommand{\removed}[1]{}

\usepackage{booktabs}
\usepackage{subcaption}
\usepackage{enumitem}
\DeclareSymbolFont{bbold}{U}{bbold}{m}{n}
\DeclareSymbolFontAlphabet{\mathbbold}{bbold}

\newenvironment{proofsketch}{\noindent\textbf{Proof sketch} \ }{\hfill$\blacksquare$}
\usepackage{times}
\usepackage{algorithm}
\usepackage{algorithmic}
\usepackage{colortbl} 
\usepackage{makecell}  
\usepackage{tabstackengine} 
\usepackage{adjustbox} 
\usepackage{titletoc} 
\usepackage{chngcntr} 

\renewcommand*{\backref}[1]{}
\renewcommand*{\backrefalt}[4]{%
    \ifcase #1 (Not cited.)%
    \or        (Cited on page~#2.)%
    \else      (Cited on pages~#2.)%
    \fi}

\coltauthor{%
 \Name{Cedar Site Bai} \Email{bai123@purdue.edu}\\
 \addr Department of Computer Science\\
  Purdue University
 \AND
 \Name{Brian Bullins} \Email{bbullins@purdue.edu}\\
 \addr Department of Computer Science\\
  Purdue University
}

\begin{document}

\maketitle

\begin{abstract}%
   Recent advances \citep{Sherman2017Area, sidford2018coordinate, cohen2021relative} have overcome the fundamental barrier of dimension dependence in the iteration complexity of solving $\ell_\infty$ regression with first-order methods. Yet it remains unclear to what extent such acceleration can be achieved for general $\ell_p$ smooth functions. In this paper, we propose a new accelerated first-order method for convex optimization under non-Euclidean smoothness assumptions. In contrast to standard acceleration techniques, our approach uses primal-dual iterate sequences taken with respect to \emph{differing} norms, which are then coupled using an \emph{implicitly} determined interpolation parameter. For $\ell_p$ norm smooth problems in $d$ dimensions, our method provides an iteration complexity improvement of up to 
$O(d^{1-\frac{2}{p}})$ in terms of calls to a first-order oracle, thereby allowing us to circumvent long-standing barriers in accelerated non-Euclidean steepest descent. 
\end{abstract}


\section{Introduction}

Large-scale optimization tasks are a central part of modern machine learning, and many of the algorithms that find success in training these models, such as SGD \citep{robbins1951stochastic}, AdaGrad \citep{duchi2011adaptive}, and Adam \citep{kingma2014adam}, among others, build on classic approaches in (convex) optimization. One prominent example is that of \emph{momentum} \citep{polyak1964some}, and the related acceleration technique of \cite{nesterov1983method}, which use both current and previous (accumulated) gradient information to accelerate beyond the basic gradient descent method. For smooth, convex problems, this \emph{accelerated gradient descent} (AGD) method converges (in terms of optimality gap) at a rate of $O(1/T^2)$, which improves upon the basic gradient descent rate of $O(1/T)$, and this rate is furthermore known to be tight due to matching lower bounds \citep{nesterov2018lectures}.

The general iterative scheme when moving from $x_t$ to $x_{t+1}$ for heavy ball momentum (with parameters $\alpha > 0$, $\beta \in [0,1]$) is given as
\begin{equation*}
    x_{t+1} = x_t - \alpha \grad f(x_t) - \beta (x_t - x_{t-1})
\end{equation*}
while the iterations of accelerated gradient descent can be expressed as
\begin{align*}
    y_{t+1} &= x_t - \alpha \grad f(x_t)\\
    x_{t+1} &= y_{t+1} - \beta(y_{t+1}-y_t).
\end{align*}
Crucially, there is a natural (Euclidean) interpretation of the trajectory of these updates, whereby each gradient step is slightly ``pushed'' in the direction of $x_{t-1} - x_{t}$ (respectively, $y_{t}-y_{t+1}$), with the amount of ``force'' applied depending on the choice of $\beta$. Indeed, this structure leads to an analysis whose final rate of convergence has a Euclidean-based ($\ell_2$ norm) dependence on the smoothness parameter, as well as the initial distance to the minimizer.

\citet{nesterov2005smooth} later generalized these techniques to non-Euclidean settings, introducing an \emph{estimate sequence} approach to acceleration, whose iterates are interpolations of two additional iterate sequences: one given by a steepest descent update (in the appropriate norm), and another given by the minimizer of a function that comprises a term linear in the accumulated gradients (which provide, in a sense, a certain \emph{dual} characterization, as also found in the basic analysis for, e.g., mirror descent \citep{nemirovski1983problem} and regret minimization \citep{hazan2016introduction}) along with the Bregman divergence of a distance generating function. Other interpretations of acceleration have since been presented (e.g., \citep{bubeck2015geometric}), including the linear coupling framework of \citet{allen2017linear}, which views acceleration as a certain \emph{coupling} between gradient (steepest) descent and mirror descent updates, each step of which takes the form (for some $\alpha, \gamma > 0$, $\beta \in [0,1]$):
\begin{align*}
    x_{t+1} &=  \beta z_t + (1-\beta)y_t\\
    y_{t+1} &= \argmin\limits_{y \in \R^d}\left\{\inner{\grad f(x_{t+1})}{y - x_{t+1}}+\frac{1}{2\alpha}\norm{y - x_{t+1}}^2\right\}\\
    z_{t+1} &= \argmin\limits_{z \in \R^d} \left\{\gamma\inner{\grad f(x_{t+1})}{z-z_t} + V_{z_t}(z)\right\},
\end{align*}
where $V_x(y)$ is the Bregman divergence with respect to the (distance generating) function $\phi(\cdot)$, i.e., $V_x(y) \defeq \phi(y) - \phi(x) - \inner{\grad \phi(x)}{y-x}$, and all of which further suggests a natural primal-dual interpretation of acceleration.

A notable use of these approaches occurs when optimizing over the simplex, in which case the fact that the Bregman divergence of the negative entropy function is strongly convex w.r.t.~$\ell_1$ norm (via Pinsker's inequality) suffices to yield the desired accelerated rate (see, e.g., Appendix~A in \citep{allen2017linear} for further details), and it even provides a natural means of deriving the (smooth) softmax approximation \citep{nesterov2005smooth, beck2012smoothing}.

Turning to the problem of $\ell_\infty$ regression \citep{kelner2014almost} or its softmax approximation \citep{nesterov2005smooth}, which is smooth with respect to the $\ell_\infty$ norm, a key challenge arises. Specifically, while the algorithm requires $\phi(\cdot)$ to be strongly convex w.r.t.~the $\ell_\infty$ norm (for the convergence guarantees to hold), \emph{any such $\phi(\cdot)$ that satisfies this condition will have a range of at least $O(d)$} \citep[Appendix A.1]{sidford2018coordinate}. While various approaches have been proposed \citep{Sherman2017Area, sidford2018coordinate, cohen2021relative} to overcome this fundamental barrier of dimension dependence in the iteration complexity in this special case of $p=\infty$, it remains unclear, in the more general $\ell_p$-smooth setting for $p > 2$, to what extent any such acceleration is possible without breaking the known lower bound of $\Omega(L\norm{x^*}_p^2/T^{\frac{p+2}{p}})$ \citep{guzman2015lower}. 

\paragraph{Our contributions.} In this work, we aim to circumvent this barrier in general by providing a faster accelerated non-Euclidean steepest descent method, called Hyper-Accelerated Steepest Descent (HASD) (Algorithm~\ref{alg:hasd}), that improves upon previous results in accelerated steepest descent \citep{nesterov2005smooth, allen2017linear} by a factor of up to $O(d^{1-\frac{2}{p}})$, in terms of calls to a first-order oracle, where $d$ is the problem dimension. (The full presentation of our convergence guarantees may be found in Theorem~\ref{thm:main}.) Our approach is based on a similar estimate sequence-type approaches as in \citep{nesterov2005smooth}, though with a key difference: rather than setting the interpolation parameter as a (fixed) function of the iteration index $t$, we instead choose the parameter implicitly, in a manner depending on \emph{local} properties of the function (specifically, the gradient at the subsequent iterate, which itself depends on the choice of parameter). 

In this way, our approach exhibits certain similarities with that of the (accelerated) HPE framework \citep{monteiro2010complexity, monteiro2013accelerated}, though \emph{we emphasize a crucial difference in the breaking of primal-dual symmetry and its eventual adjustment}. (We refer the reader to Section~\ref{sec:discussion} for a detailed discussion of this matter.) We further believe our results complement the view of A-HPE as (approximate) proximal point acceleration (e.g., \citep{carmon2020acceleration}) by offering a more general perspective in terms of primal-dual asymmetry, the algorithmic ``compensation'' for which leads to improved convergence guarantees. In addition, our analysis offers a principled framework for analyzing the oracle complexity of minimizing smooth convex functions under \textit{nonstandard geometries}---where the regularity of the objective is measured in norms \textit{differing from} the feasible domain---a longstanding open problem posed in  \citep{open2015Guzman}.

\subsection{Related work}

\paragraph{Steepest descent and acceleration.} The (unnormalized) steepest descent direction of $f$ at $x \in \R^d$ w.r.t.~a general norm $\norm{\cdot}$ is given by $\delta_{\text{sd}}(x) \defeq \norm{\grad f(x)}_* \delta_{\text{nsd}}(x)$, where we define $\delta_{\text{nsd}}(x) \defeq \argmin_{v : \norm{v} \leq 1} \grad f(x)^\top v$ \citep{boyd2004convex}, and so it follows that the gradient descent direction is a special case when taken w.r.t.~the Euclidean norm, i.e., $\norm{\cdot}_2$. Under appropriate $L$-smoothness assumptions w.r.t.~$\norm{\cdot}_2$, gradient descent (initialized at $x_0 \in \R^d$) may be further accelerated \citep{nesterov1983method}, improving the rate of convergence from $O(L\norm{x_0-x^*}_2^2/T)$ to $O(L\norm{x_0-x^*}_2^2/T^2)$, whereby the latter matches known lower bounds \citep{nesterov2018lectures}. Meanwhile, steepest descent w.r.t.~$\norm{\cdot}_p$ can be shown (under $L$-smoothness w.r.t.~$\norm{\cdot}_p$) to converge at a rate of $O(LR_p^2/T)$ (e.g., \citep{kelner2014almost}), where $R_p$ represents a bound (in terms of $\norm{\cdot}_p$) on the diameter of the problem, thereby functioning in a manner similar to the $\norm{x_0-x^*}_2^2$ term in the gradient descent rates.

Accelerated first-order methods have since been extended to non-Euclidean smoothness settings \citep{nesterov2005smooth} (see also \citep{allen2017linear} for further details). As discussed, however, the techincal requirements of these approaches lead to $\norm{\cdot}_2$ dependence in the problem diameter for $p > 2$. On the other hand, previous work by \cite{nemirovskii1985optimal} has shown how to achieve convergence rates of $O(LR_p^2/T^{\frac{p+2}{p}})$ for $p \geq 2$ (which trades off between the norm measuring the diameter and the power of $T$), and these are also known to be tight \citep{guzman2015lower, diakonikolas2024complementary}.

Also of note is the appearance of (momentum-based) steepest descent methods in the context of deep learning and training Large Language Models \citep{bernstein2018signsgd, balles2020geometry, chen2023symbolic}, for which we believe our results may provide an alternative (practical) viewpoint to acceleration for steepest descent methods. (We discuss this topic further in Section~\ref{sec:discussion}.)

\paragraph{Optimal higher-order acceleration.}
Although in this work we are primarily interested in first-order methods, some aspects of our technical contributions share similarities with recent advances in acceleration for \emph{higher-order} methods, i.e., methods which employ derivative information beyond first-order. We first recall that cubic regularization \citep{nesterov2006cubic} was shown to be amenable to (generalized) acceleration techniques \citep{nesterov2008accelerating}, and such techniques extend to $k^{th}$-order methods (that is, methods which involve minimizing a regularized $k^{th}$ order Taylor expansion), for $k > 2$ \citep{baes2009estimate}, achieving a rate of $O(1/T^{k+1})$. This has since been improved to $O(1/T^{\frac{3k+1}{2}})$ \citep{monteiro2013accelerated, gasnikov2019near} (the key idea behind which we discuss in Section~\ref{sec:discussion}), and furthermore these rates have been shown to be tight \citep{agarwal2018lower, arjevani2019oracle}.

\subsection{Outline}
We begin by establishing our setting and assumptions in Section~\ref{sec:prelims}, along with a general discussion of the particulars that occur when working with non-Euclidean norms. In Section~\ref{sec:main}, we present our main algorithm (Algorithm~\ref{alg:hasd}), as well as the key convergence guarantees (including Theorem~\ref{thm:main}) and their proofs. Section~\ref{sec:extensions} presents natural extensions of our approach to additional settings, such as strongly convex and gradient norm minimization. Finally, we conclude with a discussion of our method
as well as the opportunities presented for future work, in Sections~\ref{sec:discussion} and \ref{sec:conclusion}, respectively.

\section{Preliminaries}\label{sec:prelims}
In this work, we consider the unconstrained convex minimization problem:
\begin{equation}
    \min\limits_{x \in \R^d} f(x),
\end{equation}
where we let $x^*$ denote the minimizer of $f$. Letting $\norm{\cdot}_p$ and $\norm{\cdot}_{p^*}$ denote the standard $\ell_p$ norm and its dual norm, respectively, we are interested in the case where $f$ is $L$-\emph{smooth} w.r.t.~$\norm{\cdot}_p$, i.e., $\forall \ x, y \in \R^d$,
\begin{equation}\label{eq:smooth}
    \norm{\grad f(y) - \grad f(x)}_{p^*} \leq L\norm{y - x}_p\quad.
\end{equation}

Throughout, we specify $\norm{\cdot}_2$ when referring to the Euclidean norm, and we may observe that \eqref{eq:smooth} captures the standard (Euclidean) notion of smoothness for $p=2$. We will consider only the case where $p \geq 2$, and we further use the notation $x[i]$ to refer to the $i^{th}$ coordinate of $x$.

In addition, we say $f$ is $\mu$-\emph{strongly convex} w.r.t.~$\norm{\cdot}_p$ if, for all $x, y \in \R^d$,
\begin{equation}
    f(y) - f(x) - \inner{\grad f(x)}{y-x} \geq \frac{\mu}{2}\norm{y-x}_p^2\quad.
\end{equation}

\subsection{Comparing smoothness parameters}
It is useful to observe that $L$-smoothness w.r.t.~$\norm{\cdot}_p$ implies $L$-smoothness w.r.t.~$\norm{\cdot}_q$ for $2 \leq q \leq p$, since, by standard norm inequalities,
\begin{equation}\label{eq:smooth_alt}
    \norm{\grad f(y) - \grad f(x)}_{q^*} \leq \norm{\grad f(y) - \grad f(x)}_{p^*} \leq L\norm{y - x}_p \leq L\norm{y - x}_q\quad.
\end{equation}
On the other hand, $L$-smoothness w.r.t.~$\norm{\cdot}_q$ (for $2 \leq q \leq p$) implies $(d^{\frac{2}{q}-\frac{2}{p}}L)$-smoothness w.r.t.~$\norm{\cdot}_p$, i.e., $$\norm{\grad f(y) - \grad f(x)}_{p^*} \leq d^{\frac{2}{q}-\frac{2}{p}}L\norm{y - x}_p,$$ since
\begin{align*}
    \frac{1}{d^{\frac{1}{q}-\frac{1}{p}}}&\norm{\grad f(y) - \grad f(x)}_{p^*} \leq \norm{\grad f(y) - \grad f(x)}_{q^*} \leq L\norm{y - x}_q \leq d^{\frac{1}{q}-\frac{1}{p}}L\norm{y - x}_p.
\end{align*}

\section{Main results}\label{sec:main}
In this section, we present the main results of our paper, starting with our algorithm.
\subsection{Algorithm}
Our algorithm, called Hyper-Accelerated Steepest Descent (HASD) (Algorithm~\ref{alg:hasd}), is inspired by the estimate sequence-based approaches to acceleration (e.g., \citep{nesterov2005smooth, nesterov2008accelerating}). The key difference to observe, however, is the addition of the (simultaneous) finding of $\rho_t$ and $x_{t+1}$ such that the conditions outlined in the algorithm hold. We would further note that a line search procedure similar to \citep{bubeck2019near} can be used to find such a satisfying pair of $\rho_t$ and $x_{t+1}$, although the implicit relationship between $\rho_t$ and $x_{t+1}$ is (crucially) different from that in, for example, \citep{monteiro2013accelerated, bubeck2019near}, which requires some technical modifications we later elaborate in Section \ref{sec:binarysearch}.

\begin{algorithm}[H]
	\caption{\textbf{H}yper-\textbf{A}ccelerated \textbf{S}teepest \textbf{D}escent \textbf{(HASD)}}\label{alg:hasd}
	\begin{algorithmic}[1]
		\renewcommand{\algorithmicrequire}{\textbf{Input:}} \REQUIRE $x_0 \in \R^d$, $A_0 = 0$. Define $\psi_0(x) \defeq \frac{1}{2}\norm{x-x_0}_2^{2}$.
		\FOR{$t=0$ {\bfseries to} $T-1$}
		\STATE $v_t = \argmin\limits_{x\in \R^d} \psi_t(x)$
		\STATE Determine $\rho_t > 0$, $x_{t+1} \in \R^d$ for which the following hold simultaneously: 
        \begin{itemize}[topsep=5pt, itemsep=-2pt] 
      \item $\frac{1}{2}\frac{\norm{\grad f(x_{t+1})}_2^{2}}{\norm{\grad f(x_{t+1})}_{p^*}^{2}} \leq \rho_t \leq 2\frac{\norm{\grad f(x_{t+1})}_2^{2}}{\norm{\grad f(x_{t+1})}_{p^*}^{2}}$ 
      \item $a_{t+1} > 0$ s.t. $a_{t+1}^{2} = \frac{(A_t + a_{t+1})}{18L\rho_t}$ 
      \item Set $A_{t+1} = A_t + a_{t+1}$, $\tau_t = \frac{a_{t+1}}{A_{t+1}}$ 
      \item $y_{t} = (1-\tau_t) x_t + \tau_t v_t$ 
      \item $x_{t+1} = \argmin\limits_{x \in \R^d} \crl{\inner{\grad f(y_{t})}{x-y_t}+L\norm{x-y_t}_p^2}$
  \end{itemize} 
		\STATE $\psi_{t+1}(x) = \psi_t(x) + a_{t+1}[f(x_{t+1}) + \inner{\grad f(x_{t+1})}{x-x_{t+1}}]$
		\ENDFOR
        \renewcommand{\algorithmicrequire}{\textbf{Output:}} \REQUIRE $x_T$
	\end{algorithmic}
\end{algorithm}

\subsection{Convergence results}
We now establish the main theoretical guarantees of our work.
To begin, we show the following lemma, which establishes our basic guarantee, analogous to the standard progress guarantee, as in the case of smooth optimization. Rather than measuring how much progress we make in a single step, however, we require a bound on an term relating to the gradient \emph{at the subsequent point}.
\begin{lemma}\label{lem:progress}
    Let $f$ be $L$-smooth w.r.t.~$\norm{\cdot}_p$. Consider the update:
    \begin{equation*}
x_{t+1} = \argmin\limits_{x \in \R^d} \crl{\inner{\grad f(y_{t})}{x-y_t}+L\norm{x-y_t}_p^2}.
    \end{equation*}
    Then, $x_{t+1}$ can be expressed in closed form as
        \begin{equation}
        x_{t+1} = y_t - \frac{1}{2L}\norm{\grad f(y_t)}_{p^*}^{\frac{p-2}{p-1}} g_t \quad \text{where}\ g_t[i] \defeq \frac{\grad f(y_t)[i]}{\abs{\grad f(y_t)[i]}^{\frac{p-2}{p-1}}} \ \forall i \in \{1,\dots,d\}. \label{eq:update}
    \end{equation}

        In addition, we have that
    \begin{equation}
        \inner{\grad f(x_{t+1})}{y_t - x_{t+1}} \geq L\norm{x_{t+1}-y_t}_p^2 \geq \frac{1}{9L}\norm{\grad f(x_{t+1})}_{p^*}^2 \ . \label{eq:gradnorm}
    \end{equation}
\end{lemma}
\begin{proof} \ 
    First, we note that, by first-order optimality conditions, for all $i \in \{1,\dots,d\}$,
    \begin{equation}
        \nabla f(y_t)[i] = -2L\norm{x_{t+1}-y_t}_p^{2-p}\abs{x_{t+1}[i] - y_t[i]}^{p-2}(x_{t+1}[i]-y_t[i]).
    \end{equation}
    It may be checked by verification that $x_{t+1} - y_t = - \frac{1}{2L}\norm{\grad f(y_t)}_{p^*}^{\frac{p-2}{p-1}} g_t$ from the update in Eq. \eqref{eq:update} indeed satisfies the optimality condition. Furthermore,
    \begin{align*}\langle\grad f(x_{t+1}),& y_t - x_{t+1}\rangle = \inner{\grad f(x_{t+1}) - \grad f(y_t) + \grad f(y_t)}{y_t - x_{t+1}} \\
    &= \inner{\grad f(x_{t+1}) - \grad f(y_t)}{y_t - x_{t+1}} + \inner{\grad f(y_t)}{y_t - x_{t+1}}\\
    &= \inner{\grad f(x_{t+1}) - \grad f(y_t)}{y_t - x_{t+1}} \\
    &\qquad+ 2L\norm{x_{t+1}-y_t}_p^{2-p}\sum\limits_{i=1}^d \abs{x_{t+1}[i] - y_t[i]}^{p-2}(x_{t+1}[i]-y_t[i])^2\\
        &= \inner{\grad f(x_{t+1}) - \grad f(y_t)}{y_t - x_{t+1}} + 2L\norm{x_{t+1}-y_t}_p^{2-p}\norm{x_{t+1}-y_t}_p^{p}\\
                &= \inner{\grad f(x_{t+1}) - \grad f(y_t)}{y_t - x_{t+1}} + 2L\norm{x_{t+1}-y_t}_p^2 \ .
    \end{align*}
    Next, we observe that
    \begin{align*}
        \inner{\grad f(x_{t+1}) - \grad f(y_t)}{x_{t+1} - y_{t}} &\leq \norm{\grad f(x_{t+1}) - \grad f(y_t)}_{p^*}\norm{x_{t+1} - y_{t}}_p\\
        &\leq L\norm{x_{t+1} - y_{t}}_p^2 \ ,
    \end{align*}
    which implies that $\inner{\grad f(x_{t+1}) - \grad f(y_t)}{y_t - x_{t+1}} \geq - L\norm{x_{t+1} - y_{t}}_p^2 \ $.

    Combining this with the expression from before, it follows that
    \begin{equation*}
        \langle\grad f(x_{t+1}), y_t - x_{t+1}\rangle \geq - L\norm{x_{t+1} - y_{t}}_p^2 + 2L\norm{x_{t+1}-y_t}_p^2 = L\norm{x_{t+1} - y_{t}}_p^2 \ .
    \end{equation*}

    Using again the fact that, by first-order optimality conditions, we have, for all $i \in \{1,\dots,d\}$,
        \begin{equation}
        \nabla f(y_t)[i] + 2L\norm{x_{t+1}-y_t}_p^{2-p}\abs{x_{t+1}[i] - y_t[i]}^{p-2}(x_{t+1}[i]-y_t[i]) = 0,
    \end{equation}
    it follows that, letting $\delta$ be such that $\delta[i] = \abs{x_{t+1}[i] - y_t[i]}^{p-2}(x_{t+1}[i]-y_t[i])$,
    \begin{align*}
        \norm{\grad f(x_{t+1})}_{p^*}^2 &= \norm{\grad f(x_{t+1}) - \grad f(y_t) - 2 L\norm{x_{t+1}-y_t}_p^{2-p}\delta}_{p^*}^2\\
        &\leq \pa{\norm{\grad f(x_{t+1}) - \grad f(y_t)}_{p^*} + \norm{2L\norm{x_{t+1}-y_t}_p^{2-p}\delta}_{p^*}}^2\\
        &\leq \norm{\grad f(x_{t+1}) - \grad f(y_t)}_{p^*}^2 + \norm{2L\norm{x_{t+1}-y_t}_p^{2-p}\delta}_{p^*}^2 \\
        &\quad + 2\norm{\grad f(x_{t+1}) - \grad f(y_t)}_{p^*}\norm{2L\norm{x_{t+1}-y_t}_p^{2-p}\delta}_{p^*}\\
        &\leq L^2\norm{x_{t+1} - y_t}_{p}^2 + 4L^2\norm{x_{t+1}-y_t}_p^{2(2-p)}\norm{\delta}_{p^*}^2 + 4L^2\norm{x_{t+1}-y_t}_p^{3-p}\norm{\delta}_{p^*}\\
        &= 9L^2\norm{x_{t+1} - y_t}_{p}^2\quad,
    \end{align*}
    where the final equality used the fact that
    \begin{align*}
        \norm{\delta}_{p^*}^2 = \pa{\sum\limits_{i=1}^d \pa{\abs{x_{t+1}[i]-y_t[i]}^{p-1}}^{\frac{p}{p-1}}}^{\frac{2(p-1)}{p}} = \pa{\norm{x_{t+1}-y_t}_p^p}^{\frac{2(p-1)}{p}} = \norm{x_{t+1}-y_t}_p^{2(p-1)}.
    \end{align*}
    Combining these, it follows that $\langle\grad f(x_{t+1}), y_t - x_{t+1}\rangle \geq L\norm{x_{t+1} - y_{t}}_p^2 \geq \frac{1}{9L} \norm{\grad f(x_{t+1})}_{p^*}^2 $.
\end{proof}

Next, we proceed via the estimate sequence analysis (as in, e.g., \citep{nesterov2018lectures}, Section~4.3), adjusting for our per-step descent guarantee in terms of the $\ell_{p^*}$ norm. 

\begin{lemma}\label{lem:recurrence}
Consider Algorithm~\ref{alg:hasd}. $\forall \ t \geq 0$, we have for $B_t = \frac{1}{18L}\sum\limits_{i=0}^{t-1} A_{i+1}\norm{\grad f(x_{i+1})}_{p^*}^2$,
\begin{equation}
A_t f(x_t) + B_t \leq \psi_t^* \defeq \min\limits_{x \in \R^d} \psi_t(x).
\end{equation} 
\end{lemma}
\begin{proof} \ 
We proceed with a proof by induction. First we observe that for the base case $t = 0$, the inequality holds as both sides are 0. Next, suppose the inequality holds for some $t > 0$. Then, for any $x \in \R^d$, we have
\begin{align*}
\psi_{t+1}(x) &\geq \psi_t^* + \frac{1}{2}\norm{x-v_t}_2^{2} + a_{t+1}[f(x_{t+1})+\inner{\grad f(x_{t+1})}{x-x_{t+1}}]\\
& \geq A_t f(x_t) + B_t + \frac{1}{2}\norm{x-v_t}_2^{2} + a_{t+1}[f(x_{t+1})+\inner{\grad f(x_{t+1})}{x-x_{t+1}}]\\
&\geq A_{t+1} f(x_{t+1}) + B_t + \frac{1}{2}\norm{x-v_t}_2^{2} + \inner{\grad f(x_{t+1})}{A_t(x_t-x_{t+1}) + a_{t+1}(x-x_{t+1})}\\
&=A_{t+1} f(x_{t+1}) + B_t + \frac{1}{2}\norm{x-v_t}_2^{2} + \inner{\grad f(x_{t+1})}{a_{t+1}(x-v_t) + A_{t+1}(y_t-x_{t+1})}.
\end{align*}

Next, letting $m(x) \defeq \frac{1}{2}\norm{x-v_t}_2^{2} + a_{t+1}\inner{\grad f(x_{t+1})}{x-v_t}$, it follows that, for all $x \in \R^d$, $m(x) \geq -\frac{1}{2}a_{t+1}^{2}\norm{\grad f(x_{t+1})}_2^{2}$. Therefore, we may observe  that 
\begin{align*}
\psi_{t+1}^* &\geq A_{t+1} f(x_{t+1}) + B_t - \frac{1}{2}a_{t+1}^{2}\norm{\grad f(x_{t+1})}^{2} + A_{t+1}\inner{\grad f(x_{t+1})}{y_t - x_{t+1}}\\
&= A_{t+1} f(x_{t+1}) + B_t - \frac{A_{t+1}}{36L \rho_t}\norm{\grad f(x_{t+1})}_2^{2} + A_{t+1}\inner{\grad f(x_{t+1})}{y_t - x_{t+1}}\\
&\geq A_{t+1} f(x_{t+1}) + B_t - \frac{A_{t+1}}{36L \rho_t}\norm{\grad f(x_{t+1})}_2^{2} + \frac{A_{t+1}}{9L}\norm{\grad f(x_{t+1})}_{p^*}^2\\
&\geq A_{t+1} f(x_{t+1}) + B_t - \frac{A_{t+1}}{18L}\norm{\grad f(x_{t+1})}_{p^*}^{2} + \frac{A_{t+1}}{9L}\norm{\grad f(x_{t+1})}_{p^*}^2\\
& = A_{t+1} f(x_{t+1}) + B_t + \frac{A_{t+1}}{18L}\norm{\grad f(x_{t+1})}_{p^*}^{2}\\
& = A_{t+1} f(x_{t+1}) + B_{t+1},
\end{align*}
where the first equality follows from the algorithm that $a_{t+1}^2 = \frac{A_t + A_{t+1}}{18L\rho_t}$, the second inequality follows from Lemma~\ref{lem:progress}, and the last inequality follows from the fact that $\rho_t \geq \frac{1}{2}\frac{\norm{\grad f(x_{t+1})}_2^{2}}{\norm{\grad f(x_{t+1})}_{p^*}^{2}}$.
\end{proof}

\begin{lemma}\label{lem:growth}
    Let $\mathcal{G}_0 = 0$ and $\mathcal{G}_t \defeq \frac{1}{t}\sum\limits_{i=0}^{t-1}\frac{\norm{\grad f(x_{t+1})}_{p^*}}{\norm{\grad f(x_{t+1})}_2}$ for $t > 0$. Then, for all $t\geq 0$, we have 
    \begin{equation*}
    A_t^{1/2} \geq \frac{1}{18L^{1/2}}\sum\limits_{i=0}^{t-1}\frac{\norm{\grad f(x_{i+1})}_{p^*}}{\norm{\grad f(x_{i+1})}_2}  = \frac{t}{18L^{1/2}}\pa{\frac{1}{t}\sum\limits_{i=0}^{t-1}\frac{\norm{\grad f(x_{i+1})}_{p^*}}{\norm{\grad f(x_{i+1})}_2} } = \frac{\mathcal{G}_t t}{18L^{1/2}}.
    \end{equation*}
\end{lemma}
\begin{proof} \
    We proceed with a proof by induction. For $t=0$, $A_0 = 0$, and so the inequality holds. Suppose for $t > 0$, $A_t^{1/2} \geq  \frac{1}{18L^{1/2}}\sum\limits_{i=0}^{t-1}\frac{\norm{\grad f(x_{i+1})}_{p^*}}{\norm{\grad f(x_{i+1})}_2}$.
    Observe that
        \begin{equation}
        A_{t+1}^{1/2} - A_{t}^{1/2} = \frac{a_{t+1}}{A_{t+1}^{1/2} + A_{t}^{1/2}} = \frac{1}{A_{t+1}^{1/2} + A_{t}^{1/2}}\pa{\frac{A_{t+1}}{18L\rho_t}}^{1/2} \geq \frac{1}{9L^{1/2}\rho_t^{1/2}}.
    \end{equation}
    Thus, we have that 
    \begin{align*}
        A_{t+1}^{1/2} &\geq A_{t}^{1/2} + \frac{1}{9L^{1/2}\rho_t^{1/2}} \geq \frac{1}{18L^{1/2}}\sum\limits_{i=0}^{t-1}\frac{\norm{\grad f(x_{t+1})}_{p^*}}{\norm{\grad f(x_{t+1})}_2} + \frac{1}{9L^{1/2}\rho_t^{1/2}}\\
        &\geq \frac{1}{18L^{1/2}}\sum\limits_{i=0}^{t-1}\frac{\norm{\grad f(x_{i+1})}_{p^*}}{\norm{\grad f(x_{i+1})}_2} + \frac{1}{18L^{1/2}} \frac{\norm{\grad f(x_{t+1})}_{p^*}}{\norm{\grad f(x_{t+1})}_2} 
        = \frac{1}{18L^{1/2}}\sum\limits_{i=0}^{t}\frac{\norm{\grad f(x_{i+1})}_{p^*}}{\norm{\grad f(x_{i+1})}_2},
    \end{align*}
    where the second inequality follows from our inductive hypothesis, and the final inequality follows from the fact that $\rho_t \leq 2\frac{\norm{\grad f(x_{t+1})}_2^{2}}{\norm{\grad f(x_{t+1})}_{p^*}^{2}}$, which yields the desired result.
\end{proof}

We now have the requisite tools to prove the main theorem of our work.
\begin{theorem}[Main theorem]\label{thm:main}
Let $f$ be convex and $L$-smooth w.r.t.~$\norm{\cdot}_p\ $. Then, after $T > 0$ iterations, and letting $\mathcal{G} \defeq \frac{1}{T}\sum\limits_{t=0}^{T-1} \frac{\norm{\grad f(x_{t+1})}_{p^*}}{\norm{\grad f(x_{t+1})}_2}$, it holds that \emph{\textbf{HASD}} (Algorithm~\ref{alg:hasd}) outputs $x_T$ such that
\begin{equation}
    f(x_T) - f(x^*) \leq \frac{324L\norm{x_0-x^*}_2^2}{\mathcal{G}^2 T^2}.
\end{equation}
\end{theorem}
\begin{proof} \ By convexity, $\forall \ x \in \R^d$, $\forall \ t \in [T]$,
$$\psi_t(x) \leq \psi_{t-1}(x) + a_tf(x) \leq \psi_0(x) + \sum_{i=1}^t a_i f(x) = \frac{1}{2}\norm{x-x_0}_2^2 + A_t f\pa{x}.$$ Further by Lemma \ref{lem:recurrence}, $$A_T f(x_T) \leq \psi_T^\ast \leq \psi_T(x^*) \leq \frac{1}{2}\norm{x^*-x_0}_2^2 + A_T f\pa{x^*},$$ which yields $f(x_T) - f(x^*) \leq \frac{\norm{x_0-x^\ast}_2^2}{2A_T}$. Applying Lemma \ref{lem:growth} completes the proof.
\end{proof}
We would note that the key difference between the rate of Theorem~\ref{thm:main} and that as presented in, e.g., \citep{allen2017linear}, is precisely the addition of the $\mathcal{G}^2$ term, \emph{which may be as large as $d^{1-\frac{2}{p}}$} (and which is furthermore always $\geq 1$). We would also note that there exists an instance of minimizing the softmax function $f^s$, as shown in Appendix \ref{app:softmax}, for which $\mathcal{G}^2 = d^{1-\frac{2}{p}}$, thus yielding the rate, when $p = \infty$, $f^s\pa{x_T} - f^s\pa{x^*} \leq \frac{324L_s\norm{x_0-x^*}_2^2}{d T^2} \leq \frac{324L_s\norm{x_0-x^*}_\infty^2}{ T^2}$ where $L_s \in \mc{O}\pa{\eps^{-1}}$ when approximating $\ell_\infty$ regression, thereby matching, for this particular instance,\footnote{
While we include this softmax example to provide an instance where $O(d^{1-\frac{2}{p}})$ acceleration is achieved, it remains to be investigated how to further incorporate the affine transformation for more general $\ell_\infty$ regression.}the rate given by \citep{Sherman2017Area, sidford2018coordinate, cohen2021relative}. We cannot, however, achieve rates of this sort for general $\ell_p$-smooth functions, due to the lower bound \citep{guzman2015lower}.

\subsection{Complexity of Binary Search} \label{sec:binarysearch}
In this section, we characterize the complexity of binary search to simultaneously find a pair of $\rho_t$ and $x_{t+1}$ that satisfies the implicit relation of $\frac{1}{2}\frac{\norm{\grad f(x_{t+1})}_2^{2}}{\norm{\grad f(x_{t+1})}_{p^*}^{2}} \leq \rho_t \leq 2\frac{\norm{\grad f(x_{t+1})}_2^{2}}{\norm{\grad f(x_{t+1})}_{p^*}^{2}}$ at each iteration. We show in the following theorem that such binary search takes at most $\mc{O}\pa{\log\pa{d} + \log\pa{\frac{1}{\eps}}}$ calls to the first-order steepest descent oracle. The proof follows the general framework of Theorem 18 in \citep{bubeck2019near}, with several non-trivial technical modifications to accommodate the $\ell_p$ norm and the specific formulation of the condition, which involves the ratio of the $\ell_2$ and $\ell_{p^*}$ norms of the gradient. We provide a proof sketch and defer the complete proof to Appendix \ref{app:binarysearch}.

\begin{restatable}{theorem}{thmbinarysearch} \label{thm:binarysearch}
    For any iteration $t$ in Algorithm \ref{alg:hasd}, with at most $9+\frac{5(p-2)}{2p}\log_2\pa{d} + \log_2\pa{\frac{LD_R}{\eps}}$ calls to the first-order $\ell_p$ steepest descent oracle, we can find either a point $x_{t+1}$ such that $f(x_{t+1}) - f(x^\ast) \leq \eps$ or a pair of $\rho_t$, $x_{t+1}$ that satisfies the condition  $\frac{1}{2}\frac{\norm{\grad f(x_{t+1})}_2^{2}}{\norm{\grad f(x_{t+1})}_{p^*}^{2}} \leq \rho_t \leq 2\frac{\norm{\grad f(x_{t+1})}_2^{2}}{\norm{\grad f(x_{t+1})}_{p^*}^{2}}$ for $0 < \eps \leq \frac{LD_R}{6}$ where $D_R = \pa{R + 1458R^2}\pa{20R + 4374R^2}$ for $R = \norm{x_0-x^*}_2$.
\end{restatable}
\begin{proofsketch}
    Given the implicit dependence of $x_{t+1}$ on $\rho_t$ and vice versa, we make such dependence explicit by letting $\theta \coloneqq \frac{A_t}{A_{t+1}}$, so that other variables can all be expressed as a function of $\theta$, denoted as $x_\theta \coloneqq x_{t+1}$, $y_\theta \coloneqq y_t = \theta x_t + (1-\theta)v_t$, and $\rho_\theta \coloneqq \rho_t = \frac{\theta}{18L(1-\theta)^2 A_t}$. As a result, the search for $\rho_t$, $x_{t+1}$ that satisfy $\frac{1}{2}\frac{\norm{\grad f(x_{t+1})}_2^{2}}{\norm{\grad f(x_{t+1})}_{p^*}^{2}} \leq \rho_t \leq 2\frac{\norm{\grad f(x_{t+1})}_2^{2}}{\norm{\grad f(x_{t+1})}_{p^*}^{2}}$ is equivalent to finding $\theta$ such that 
    $$ \frac{1}{2} \leq \zeta\pa{\theta}\coloneqq \frac{18L(1-\theta)^2 A_t}{\theta}\frac{\norm{\grad f(x_\theta)}_2^{2}}{\norm{\grad f(x_\theta)}_{p^*}^{2}} \leq 2 .$$
    Noting that $\zeta\pa{0} = \infty$, $\zeta\pa{1} = 0$, $\exists \ \theta^\ast$ such that $\zeta\pa{\theta^\ast} = \frac{5}{4}$. Then one can use $\log_2\pa{\frac{1}{\delta}}$ binary search steps to find $\theta$ such that $\abs{\theta - \theta^\ast} \leq \delta$. It remains to verify that with certain choice of $\delta$, we indeed have $\zeta\pa{\theta} \in \bracks{\frac{1}{2}, 2}$. How the function value changes with respect to the input within some $\delta$-neighborhood is characterized by the Lipschitz constant of $\zeta\pa{\theta}$, which we analyze by showing
    $$ \abs{\frac{d}{d\theta}\log\pa{\zeta\pa{\theta}}} \leq \frac{2}{1-\theta} + \frac{1}{\theta} + \frac{4d^\frac{p-2}{2p}}{\norm{\grad f(x_\theta)}_{p^*}} \norm{\grad^2 f(x_\theta)}_p\norm{\frac{d}{d\theta}x_\theta}_p.$$
    We bound it by bounding each of the relevant terms, $\norm{\grad f(x_\theta)}_{p^*}$ by Lemma \ref{lem:grad_lower}, $\norm{\grad^2 f(x_\theta)}_p$ by Lemma \ref{lem:hessian}, and $\norm{\frac{d}{d\theta}x_\theta}_p$ by Lemma \ref{lem:d_theta}, which involves nontrivially showing for $x \in \R^d$, $\grad^2\norm{x}_p^2$ is the sum of two positive semidefinite matrices, after which it can be further simplified to  $\abs{\frac{d}{d\theta}\log\pa{\zeta\pa{\theta}}} \leq \omega\pa{\theta}\pa{1 + \frac{1}{\zeta\pa{\theta}} + \zeta\pa{\theta}}$ for $\omega\pa{\theta} = 4d^\frac{5(p-2)}{2p}\pa{6 + 9LA_t + \frac{L D_R}{f(x_\theta) - f(x^\ast)}}$ where $D_R = \pa{R + 1458R^2}\pa{20R + 4374R^2}$.
    Finally, we show in Lemma \ref{lem:zeta_range} that if the Lipschitz constant (as a function of $\theta$) is bounded as such, by properly choosing the neighborhood $\delta = \frac{\eps}{320d^\frac{5(p-2)}{2p}LD_R} \leq \frac{1}{10\omega\pa{\theta}}$, $\zeta\pa{\theta}$ falls within the range $\bracks{\frac{1}{2}, 2}$ when $\abs{\theta - \theta^\ast} \leq \delta$.
\end{proofsketch}

\section{Extensions}\label{sec:extensions}
In this section, we consider natural extensions of our method to additional related problem settings, including minimizing strongly convex objectives, as well as minimizing the $\ell_{p^*}$ norm of the gradient. The latter---explored, for example, in \citep{gratton2023adaptive, diakonikolas2024complementary}---may be of independent interest, as it allows us to deviate from the typical goal of minimizing in terms of the $\ell_2$ norm. 

\subsection{Strongly convex setting}
We begin by considering the case in which $f$ is additionally $\mu$-strongly convex, whereby combining our method with the usual restarting scheme lets us straightforwardly improve from a sublinear to a linear rate. Furthermore, it is important to note that the improvements our method offers in the smooth and (weakly) convex setting appear, in the strongly convex setting, outside of the log factor, along with the condition number of the problem. Due to space constraints, we provide the full details of our algorithm for the strongly convex setting (\textbf{HASD + Restarting}) in Appendix~\ref{app:extensions}.

Using our results from Section~\ref{sec:main}, we now arrive at the following corollary.
\begin{corollary}
    Let $\eps > 0$, let $x_{\text{outer},0} \in \R^d$, and let $K = O(\log(1/\eps))$. Consider $f$ that is $L$-smooth w.r.t.~$\norm{\cdot}_p$ and $\mu$-strongly convex w.r.t.~$\norm{\cdot}_2$. Assume that, for all $i \in \{1,\dots,K\}$, the respective average term $\mathcal{G}_i \geq \hat{\mathcal{G}} \geq 1$. Then, the method \emph{\textbf{HASD + Restarting}} (Algorithm \ref{alg:hasd_sc}) outputs $x_{\text{outer},K}$ such that
    \begin{equation}
        f(x_{outer,K}) - f(x^*) \leq \eps.
    \end{equation}
\end{corollary}
\begin{proof} \
    Note that by Theorem~\ref{thm:main}, for all $i$, it holds that
    \begin{align*}
        f(x_{\text{outer},i+1}) - f(x^*) \leq \frac{324L\norm{x_{\text{outer},i} - x^*}_2^2}{\mathcal{G}_i^2T^2}
    \end{align*}
    where $\mathcal{G}_i \defeq \frac{1}{T}\sum\limits_{t=0}^{T-1} \frac{\norm{\grad f(x_{t+1})}_{p^*}}{\norm{\grad f(x_{t+1})}_2}$ (where $x_t$ are w.r.t.~the $i^{th}$ outer iteration). By $\mu$-strong convexity, we have that $\norm{x_{\text{outer},i} - x^*}_2^2 \leq \frac{2}{\mu}(f(x_{\text{outer},i}) - f(x^*))$, and so it follows that
    \begin{equation*}
                f(x_{\text{outer},i+1}) - f(x^*) \leq \frac{648L(f(x_{\text{outer},i}) - f(x^*))}{\mu\mathcal{G}_i^2T^2}.
    \end{equation*}
    Thus, by setting $T = \frac{36}{\hat{\mathcal{G}}}\sqrt{\frac{L}{\mu}}$, we have that
        $f(x_{\text{outer},i+1}) - f(x^*) \leq \frac{f(x_{\text{outer},i}) - f(x^*)}{2}$,
    and so, since we halve the optimality gap each time, the desired result follows from the recurrence after $K = O(\log(1/\eps))$ (outer) iterations.
\end{proof}

Interestingly, due to the diameter term being $\norm{x_0-x^*}_2^2$ (i.e., in terms of the $\ell_2$ norm) in the final convergence expression for the smooth and (weakly) convex case, we similarly need the strong convexity assumption to be w.r.t.~$\ell_2$ for the analysis to hold.

\subsection{Gradient norm minimization}
As a natural consequence of the results in both the (weakly) convex and strongly convex setting, we may additionally derive guarantees in terms of minimizing the $\ell_{p^*}$ norm of the gradient, as has been considered (in the case of the $\ell_2$ norm) in both convex (e.g., \citep{nesterov2012make, allen2018make}) and non-convex (e.g., \citep{agarwal2017finding, carmon2018accelerated}) settings.

\subsubsection{First attempt: directly relating to optimality gap}
Aiming to minimize the gradient norm, we begin with the following lemma, the proof of which follows the standard transition from optimality gap to gradient norm by convexity and smoothness and may be found in Appendix \ref{app:gradnormlem}. 

\begin{restatable}{lemma}{lemgradnormmin}\label{lem:gradnormmin}
    Let $f$ be convex and $L$-smooth w.r.t.~$\norm{\cdot}_p$, and let $x \in \R^d$ be such that $f(x) - f(x^*) \leq \frac{\eps^2}{2L}.$ Then, $\norm{\grad f(x)}_{p^*} \leq \eps$.
\end{restatable}
\noindent Combining this with our main convergence guarantee (Theorem~\ref{thm:main}) leads to the following corollary.
\begin{corollary}
    Let $R > 0$ be such that $\norm{x_0 - x^*}_2 \leq R$, and let $x_T$ be the output of \emph{\textbf{HASD}} (Algorithm~\ref{alg:hasd}) after $T = \left\lceil\frac{18\sqrt{2}LR}{\hat{\mathcal{G}}\eps}\right\rceil$ iterations, where $\hat{\mathcal{G}}$ is such that $\mathcal{G} \geq \hat{\mathcal{G}} \geq 1$. Then,
    \begin{equation*}
        \norm{\grad f(x_T)}_{p^*} \leq \eps.
    \end{equation*}
\end{corollary}

\subsubsection{Improved rate: using additional $B_t$ term}
By observing more closely the recurrence relation established in Lemma~\ref{lem:recurrence}, we may further improve the rate for minimizing $\norm{\grad f(x)}_{p^*}$.
\begin{corollary}
    Let $R > 0$ be such that $\norm{x_0 - x^*}_2 \leq R$, and let $x_t$ ($t \in \{1,\dots T\}$) be the iterates generated by \emph{\textbf{HASD}} (Algorithm~\ref{alg:hasd}) after 
    $T = \left\lceil\frac{21 L^{\frac{2}{3}}R^{\frac{2}{3}}}{\hat{\mathcal{G}}^{\frac{2}{3}}\eps^{\frac{2}{3}}}\right\rceil$ iterations, where $\hat{\mathcal{G}}$ is such that, for all $t \in \{1,\dots T\}$, $\mathcal{G}_t \geq \hat{\mathcal{G}} \geq 1$. Then,
    \begin{equation*}
        \min\limits_{t \in \{1,\dots T\}} \norm{\grad f(x_t)}_{p^*} \leq \eps.
    \end{equation*}
\end{corollary}
\begin{proof} \
    The corollary follows from the fact that, by Lemma~\ref{lem:recurrence}, we know 
    \begin{align*}
        \pa{\frac{1}{18L}\sum_{t=0}^{T-1}A_{t+1}}\pa{\min\limits_{t \in \{1,\dots T\}} \norm{\grad f(x_t)}_{p^*}^2} \leq B_T = \frac{1}{18L}\sum\limits_{t=0}^{T-1} A_{t+1}\norm{\grad f(x_{t+1})}_{p^*}^2 \leq \frac{1}{2}\norm{x_0 - x^*}_2^2.
    \end{align*}
    Therefore, given $A_t^{1/2} \geq  \frac{\mathcal{G}_t t}{18L^{1/2}}$ from Lemma \ref{lem:growth},
    \begin{align*}
        \min\limits_{t \in \{1,\dots T\}} \norm{\grad f(x_t)}_{p^*}^2 \leq \frac{9LR^2}{\sum_{t=0}^{T-1}A_{t+1}} 
        \leq \frac{2916L^2R^2}{\sum_{t=1}^{T} \mathcal{G}_t^2 t} 
        \leq \frac{2916L^2R^2}{\hat{\mathcal{G}}^2 \frac{T(T+1)(2T+1)}{6}} 
        \leq \frac{8748L^2R^2}{\hat{\mathcal{G}}^2 T^3}
    \end{align*}
    From $\frac{8748L^2R^2}{\hat{\mathcal{G}}^2 T^3} \leq \epsilon^2$ we solve for $T\geq \frac{21 L^{\frac{2}{3}}R^{\frac{2}{3}}}{\hat{\mathcal{G}}^{\frac{2}{3}}\eps^{\frac{2}{3}}}$.
\end{proof}

\section{Discussion}\label{sec:discussion}
We believe the relative simplicity of our algorithm affords it the opportunity to be expanded and simplified (from both theoretical \emph{and} practical perspectives), as it is a readily implementable and efficient first-order algorithm. Recent developments in sign (stochastic) gradient methods (which are a special case of steepest descent w.r.t.~$\norm{\cdot}_{\infty}$) \citep{bernstein2018signsgd, balles2020geometry, chen2023symbolic} highlight the critical importance of better understanding the interplay between acceleration and steepest descent, and we feel that our work brings additional insights and perspectives to this context. 

\paragraph{Comparison with \citep{monteiro2013accelerated}.} In contrast to work by \cite{nesterov2008accelerating}, which achieves a rate of $O(1/T^3)$ for accelerated cubic regularization, \cite{monteiro2013accelerated} establish an improved $\tilde{O}(1/T^{7/2})$ rate for the same (convex, second-order smooth) setting, by using the Accelerated Hybrid Proximal Extragradient (A-HPE) method. A key algorithmic difference between the two approaches lies in certain choice of regularization function. Namely, whereas \cite{nesterov2008accelerating} uses $\frac{1}{3}\norm{x-x_0}_2^3$---here, we note that the exponent is $3$, matching that of the cubic regularization term and thus \emph{maintaining a certain symmetry}---\cite{monteiro2013accelerated} instead use a \emph{quadratic} term, thereby \emph{breaking the symmetry}. Much of the subsequent analysis (which gives the near-optimal rate) is based around \emph{adjusting for this broken symmetry}.

We wish to emphasize this (high-level) observation, as our approach follows a similar means of ``symmetry-breaking,'' though instead of doing so w.r.t.~the exponent of the regularizer, we do so \emph{w.r.t.~the norm itself}. Specifically, we combine $\ell_p$ norm-based ($p \geq 2$) steepest descent steps with $\ell_2$ norm-based mirror descent steps, \emph{and algorithmically adjust for this discrepancy}. It follows that, in appropriately accounting for this adjustment, our algorithm HASD yields additional convergence gains (as made explicit in Theorem~\ref{thm:main}) \emph{beyond} what has been previously shown.

\paragraph{Practical considerations.} 
We provide in Appendix~\ref{app:experiments} preliminary experimental evidence to validate (and complement) the main guarantees of our algorithm. We would also note that, while a line search procedure is needed in theory, in practice this may likely be relaxed, by using, e.g., a heuristic procedure. Furthermore, it may be the case that the line search could be (effectively) removed altogether, as has been shown in the case of high-order optimization \citep{carmon2022optimal, kovalev2022first}, and doing so may provide an interesting future research direction.

\section{Conclusion and future work}\label{sec:conclusion}
We have presented a new method for accelerating non-Euclidean steepest descent, based on an implicit interpolation of steepest and mirror descent updates in differing norms, which offers up to $O(d^{1-\frac{2}{p}})$ improvement, in terms of iteration complexity, when considering smoothness w.r.t.~$\norm{\cdot}_p$. We believe our results suggest there are many more interesting directions yet to be explored, even in the case of \emph{first-order} acceleration. Due to the role optimization landscapes and geometry play in training deep learning models, we are also optimistic that our approach might lend itself to more practical considerations---with promising evidence of such possibilites recently presented by~\citet{luo2025stacey} for stochastic settings---though we leave a full exploration of this to future work.

Another possibility would be to consider whether we might achieve a more fine-grained analysis, as has been shown by \citet{sidford2018coordinate} for the problem of $\ell_\infty$ regression. Determining (matching) lower bounds would also help to clarify the placement of our results. In addition, given the fact that standard (Euclidean) acceleration has been extended to non-convex \citep{agarwal2017finding, carmon2018accelerated} and stochastic  \citep{ghadimi2013stochastic} settings, we believe it may be possible to extend our results in a similar manner.
As our analysis provides a principled framework for characterizing the oracle complexity of minimizing smooth convex functions under \textit{nonstandard geometries} \citep{open2015Guzman}, we suspect our techniques could also apply to more general domains whose diameters are measured in alternative $q$-norms.

\bibliography{references}

\newpage
\appendix

\section{Algorithm for Strongly Convex Setting}\label{app:extensions}
We include here the algorithm for the strongly convex setting.
\begin{algorithm}[H] 
	\caption{\textbf{HASD + Restarting}} \label{alg:hasd_sc}
	\begin{algorithmic}
        \renewcommand{\algorithmicrequire}{\textbf{Input:}} \REQUIRE $x_{\text{outer},0} \in \R^d$, $\eps > 0$, $\hat{\mathcal{G}} > 0$, $K = O(\log(1/\eps))$
        \FOR{$i=0$ {\bfseries to} $K-1$}
        \STATE $A_0 = 0$
        \STATE Define $\psi_0(x) \defeq \frac{1}{2}\norm{x-x_{i,0}}_2^{2}$.
        \STATE $x_0 = x_{outer, i}$
        \STATE $T = \frac{36}{\hat{\mathcal{G}}}\sqrt{\frac{L}{\mu}}$ 
		\FOR{$t=0$ {\bfseries to} $T-1$}
		\STATE $v_t = \argmin\limits_{x\in \R^d} \psi_t(x)$
		\STATE Determine $\rho_t > 0$, $x_{t+1} \in \R^d$ for which the following hold simultaneously:
  \STATE\begin{itemize}[leftmargin=.7in, itemsep=0pt] \vspace{4pt}
      \item $\frac{1}{2}\frac{\norm{\grad f(x_{t+1})}_2^{2}}{\norm{\grad f(x_{t+1})}_{p^*}^{2}} \leq \rho_t \leq 2\frac{\norm{\grad f(x_{t+1})}_2^{2}}{\norm{\grad f(x_{t+1})}_{p^*}^{2}}$ 
      \item $a_{t+1} > 0$ s.t. $a_{t+1}^{2} = \frac{(A_t + a_{t+1})}{18L\rho_t}$
      \item Set $A_{t+1} = A_t + a_{t+1}$, $\tau_t = \frac{a_{t+1}}{A_{t+1}}$
      \item $y_{t} = (1-\tau_t) x_t + \tau_t v_t$
      \item $x_{t+1} = \argmin\limits_{x \in \R^d} \crl{\inner{\grad f(y_{t})}{x-y_t}+L\norm{x-y_t}_p^2}$
  \end{itemize}
		\STATE $\psi_{t+1}(x) = \psi_t(x) + a_{t+1}[f(x_{t+1}) + \inner{\grad f(x_{t+1})}{x-x_{t+1}}]$
		\ENDFOR
        \STATE $x_{outer, i+1} = x_T$
        \ENDFOR
        \renewcommand{\algorithmicrequire}{\textbf{Output:}} \REQUIRE $x_{outer, K}$
	\end{algorithmic}
\end{algorithm}

\section{Example: The Softmax Function} \label{app:softmax}
Consider, as an illustration, the (symmetric) softmax objective: $$\min\limits_{x\in \mathbb{R}^d}  f^s(x) \coloneqq \alpha\log\left(\sum\limits_{i=1}^d e^{\frac{x[i]}{\alpha}} + e^{\frac{-x[i]}{\alpha}}\right) $$ For appropriate choice of $\alpha$, the softmax function closely approximates $\ell_{\infty}$ regression \citep{nesterov2005smooth}. Further, suppose we initialize at $x_0 = [1,\dots, 1]^\top \in \mathbb{R}^d$. Then we may observe that, by symmetry, the iterates of the algorithm $x_0, \dots, x_T$ satisfy $\forall \ i, j$, $\frac{\partial}{\partial x[i]}f^s(x_t) = \frac{\partial}{\partial x[j]} f^s(x_t)$ for $t = 0, \dots, T$. That is to say, $\forall \ t \in [T]$, $\frac{\Vert\nabla f^s(x_t)\Vert_{p^*}}{\Vert\nabla f^s(x_t)\Vert_2} = d^{\frac{1}{2}-\frac{1}{p}}$, and so $\mathcal{G} = d^{\frac{1}{2} - \frac{1}{p}}$.

\section{Naive Analysis for Gradient Norm Minimization} \label{app:gradnormlem}
\lemgradnormmin*
\begin{proof} \
    Let $z = x - \frac{1}{L}\norm{\grad f(x)}_{p^*}^{\frac{p-2}{p-1}}g$, where $g$ is such that $g[i] \defeq \frac{\grad f(x)[i]}{\abs{\grad f(x)[i]}^{\frac{p-2}{p-1}}}$ for all $i \in \{1,\dots,d\}$. Using our smoothness assumption, along with the fact that $\grad f(x^*) = 0$, it follows that
    \begin{align*}
        f(x^*) - f(x) &= f(x^*)-f(z) + f(z) - f(x)\\
        &\leq \inner{\grad f(x^*)}{x^*-z} + \inner{\grad f(x)}{z-x} + \frac{L}{2}\norm{z-x}_p^2\\
        &= -\frac{1}{L}\norm{\grad f(x)}_{p^*}^{\frac{p-2}{p-1}}\inner{\grad f(x)}{g} + \frac{1}{2L}\norm{\grad f(x)}_{p^*}^{\frac{2(p-2)}{p-1}}\norm{g}_p^2\\
        &= -\frac{1}{2L}\norm{\grad f(x)}_{p^*}^2,
    \end{align*}
    where the inequality follows from convexity and $L$-smoothness of $f$. Rearranging and multiplying both sides by $2L$ gives us $\norm{\grad f(x)}_{p^*}^2 \leq 2L(f(x)-f(x^*))$,
    and so, using $$f(x) - f(x^*) \leq \frac{\eps^2}{2L},$$ it follows that $$\norm{\grad f(x)}_{p^*} \leq \eps,$$ as desired.
\end{proof}

\section{Complexity of Binary Search} \label{app:binarysearch}
\subsection{Proof of Theorem \ref{thm:binarysearch}}
\thmbinarysearch*
\begin{proof} \
For every iteration of the algorithm, we seek for $\rho_t > 0$, $x_{t+1} \in \R^d$ for which the following hold simultaneously: 
        \begin{itemize}[topsep=5pt, itemsep=-2pt] 
      \item $\frac{1}{2}\frac{\norm{\grad f(x_{t+1})}_2^{2}}{\norm{\grad f(x_{t+1})}_{p^*}^{2}} \leq \rho_t \leq 2\frac{\norm{\grad f(x_{t+1})}_2^{2}}{\norm{\grad f(x_{t+1})}_{p^*}^{2}}$ 
      \item $a_{t+1} > 0$ s.t. $a_{t+1}^{2} = \frac{(A_t + a_{t+1})}{18L\rho_t}$ 
      \item Set $A_{t+1} = A_t + a_{t+1}$, $\tau_t = \frac{a_{t+1}}{A_{t+1}}$ 
      \item $y_{t} = (1-\tau_t) x_t + \tau_t v_t$ 
      \item $x_{t+1} = \argmin\limits_{x \in \R^d} \crl{\inner{\grad f(y_{t})}{x-y_t}+L\norm{x-y_t}_p^2}$
  \end{itemize}
  One can see the circular dependence of $x_{t+1}$ on $y_t$, which depends on $a_{t+1}$, which depends on $\rho_t$, which then depends on $x_{t+1}$. We break such circular dependence by letting $\theta \coloneqq \frac{A_t}{A_{t+1}}$, and express other variables as a function of $\theta$, denoted as $x_\theta \coloneqq x_{t+1}$, $y_\theta \coloneqq y_t$, and $\rho_\theta \coloneqq \rho_t$. We further denote $z_\theta \coloneqq x_\theta - y_\theta$. By definition, $\frac{a_{t+1}}{A_{t+1}} = 1-\theta$ and we have $y_\theta = \theta x_t + (1-\theta)v_t$. As a result, $\rho_\theta = \frac{A_t + a_{t+1}}{18La_{t+1}^2} = \frac{\theta}{18L(1-\theta)^2 A_t}$. All conditions are now satisfied except for the first condition, which is equivalent to 
  \begin{align*}
      \frac{1}{2} \leq \frac{\norm{\grad f(x_\theta)}_2^{2}}{\rho_\theta\norm{\grad f(x_\theta)}_{p^*}^{2}} \leq 2
  \end{align*}
  Defining 
  \begin{align*}
      \zeta \pa{\theta} \coloneqq \frac{\norm{\grad f(x_\theta)}_2^{2}}{\rho_\theta\norm{\grad f(x_\theta)}_{p^*}^{2}} = \frac{18L(1-\theta)^2 A_t}{\theta}\frac{\norm{\grad f(x_\theta)}_2^{2}}{\norm{\grad f(x_\theta)}_{p^*}^{2}},
  \end{align*}
   we can search for $\theta$ such that $\frac{1}{2} \leq \zeta\pa{\theta} \leq 2$, which then yields all conditions satisfied simultaneously.
   
Given that $\zeta\pa{0} = \infty$, $\zeta\pa{1} = 0$, $\exists \ \theta^\ast \in [0,1]$ such that $\zeta\pa{\theta^\ast} = \frac{5}{4}$. Then one can use $\log_2\pa{\frac{1}{\delta}}$ bineary search step to find $\theta$ such that $\abs{\theta - \theta^\ast} \leq \delta$. Now we verify that with certain choice of $\delta$, $\zeta\pa{\theta} \in \bracks{\frac{1}{2}, 2}$. How the function value changes with respect to the input within some $\delta$-neighborhood is characterized by the Lipschitz constant of $\zeta\pa{\theta}$, i.e., the upper bound on $\abs{\frac{d}{d\theta}\zeta\pa{\theta}}$. For simplicity, we start by analyzing $\abs{\frac{d}{d\theta}\log\pa{\zeta\pa{\theta}}}$. It's trivial that 
\begin{align*}
       \log\pa{\zeta\pa{\theta}} = 2\log\pa{1-\theta} - \log\pa{\theta} + \log\pa{18LA_t} + 2\log\pa{\norm{\grad f(x_\theta)}_2} - 2 \log\pa{\norm{\grad f(x_\theta)}_{p^*}}.
\end{align*}
Taking the derivative, we have
\begin{align*}
       \frac{d}{d\theta}\log\pa{\zeta\pa{\theta}} &= -\frac{2}{1-\theta} - \frac{1}{\theta} + \frac{2}{\norm{\grad f(x_\theta)}_2^2}\grad^2 f(x_\theta)\grad f(x_\theta) \frac{d}{d\theta}x_\theta \\ 
       &\quad - \frac{2}{\norm{\grad f(x_\theta)}_{p^*}^{p^*}} \grad^2 f(x_\theta)\begin{bmatrix}
           \abs{\grad f(x_\theta)[1]}^{p^*-2}\grad f(x_\theta)[1] \\
           \vdots \\
           \abs{\grad f(x_\theta)[d]}^{p^*-2}\grad f(x_\theta)[d]
         \end{bmatrix}\frac{d}{d\theta}x_\theta.
   \end{align*}
Taking the $\ell_p$ norm on both sides, we have
\begin{align*}
        \abs{\frac{d}{d\theta}\log\pa{\zeta\pa{\theta}}} &\leq \frac{2}{1-\theta} + \frac{1}{\theta} + \frac{2}{\norm{\grad f(x_\theta)}_2^2}\norm{\grad^2 f(x_\theta)}_p\norm{\grad f(x_\theta)}_p \norm{\frac{d}{d\theta}x_\theta}_p \\ 
       &\quad + \frac{2}{\norm{\grad f(x_\theta)}_{p^*}} \norm{\grad^2 f(x_\theta)}_p\norm{\frac{d}{d\theta}x_\theta}_p \\
       &\leq \frac{2}{1-\theta} + \frac{1}{\theta} + \frac{4}{\norm{\grad f(x_\theta)}_2} \norm{\grad^2 f(x_\theta)}_p\norm{\frac{d}{d\theta}x_\theta}_p\\
       &\leq \frac{2}{1-\theta} + \frac{1}{\theta} + \frac{4d^\frac{p-2}{2p}}{\norm{\grad f(x_\theta)}_{p^*}} \norm{\grad^2 f(x_\theta)}_p\norm{\frac{d}{d\theta}x_\theta}_p.
    \end{align*}
For the first two terms, we have
 \begin{align*}
       \frac{1}{1-\theta} \leq 1 + \frac{\theta}{(1-\theta)^2} = 1 + \frac{18LA_t}{\zeta\pa{\theta}} \frac{\norm{\grad f(x_\theta)}_2^{2}}{\norm{\grad f(x_\theta)}_{p^*}^{2}} \leq 1 + \frac{18LA_t}{\zeta\pa{\theta}},
   \end{align*}
and
   \begin{align*}
       \frac{1}{\theta} \leq 2 + \frac{(1-\theta)^2}{\theta} = 2 + \frac{\zeta\pa{\theta}}{18LA_t}\frac{\norm{\grad f(x_\theta)}_{p^*}^{2}}{\norm{\grad f(x_\theta)}_2^{2}} \leq 2 + \frac{d^\frac{p-2}{p}\zeta\pa{\theta}}{18LA_t}.
   \end{align*}
    For each of the relevant components in the third term, we lower bound $\norm{\grad f(x_\theta)}_{p^*}$ in Lemma \ref{lem:grad_lower}, upper bound $\norm{\grad^2 f(x_\theta)}_p$ and $\norm{\frac{d}{d\theta}x_\theta}_p$ by Lemma \ref{lem:hessian},  and Lemma \ref{lem:d_theta}. With these results, we have
   \begin{align*}
        \abs{\frac{d}{d\theta}\log\pa{\zeta\pa{\theta}}} &\leq \frac{2}{1-\theta} + \frac{1}{\theta} + \frac{4 d^\frac{p-2}{2p} \norm{\grad^2 f(x_\theta)}_p}{\norm{\grad f(x_\theta)}_{p^*}} \norm{\frac{d}{d\theta}x_\theta}_p \\
        &\leq 4 + \frac{36LA_t}{\zeta\pa{\theta}} + \frac{d^\frac{p-2}{p}\zeta\pa{\theta}}{18LA_t} + \frac{4 d^\frac{p-2}{2p} \norm{\grad^2 f(x_\theta)}_p}{\norm{\grad f(x_\theta)}_{p^*}} \norm{\frac{d}{d\theta}x_\theta}_p \\
        &\leq 4 + \frac{36LA_t}{\zeta\pa{\theta}} + \frac{d^\frac{p-2}{p}\zeta\pa{\theta}}{18LA_t} + \frac{4 L \pa{R + 1458R^2}d^\frac{3(p-2)}{2p}}{\frac{f(x_\theta) - f(x^\ast)}{\pa{19+d^\frac{p-2}{p}}R + \pa{2916 + 1458d^\frac{p-2}{p}}R^2}}  & \text{(Lemma \ref{lem:d_theta}, \ref{lem:grad_lower}, \ref{lem:hessian})}\\
        &\leq 4 + \frac{36LA_t}{\zeta\pa{\theta}} + \frac{d^\frac{p-2}{p}\zeta\pa{\theta}}{18LA_t} + \frac{4 L D_Rd^\frac{5(p-2)}{2p}}{f(x_\theta) - f(x^\ast)} \\
        &\leq 4 + \frac{36LA_t}{\zeta\pa{\theta}} + 18d^\frac{p-2}{p}\zeta\pa{\theta} + \frac{4 L D_Rd^\frac{5(p-2)}{3p}}{f(x_\theta) - f(x^\ast)} & \text{(Lemma \ref{lem:growth})} \\
        &\leq \omega\pa{\theta}\pa{1 + \frac{1}{\zeta\pa{\theta}} + \zeta\pa{\theta}}
    \end{align*}
    for $\omega\pa{\theta} = 4d^\frac{5(p-2)}{2p}\pa{6 + 9LA_t + \frac{L D_R}{f(x_\theta) - f(x^\ast)}}$ in which $D_R = \pa{R + 1458R^2}\pa{20R + 4374R^2}$.

Now we choose the proper $\delta$. First, we claim that $A_t \leq \frac{R^2}{2\eps}$ and $f(x_\theta) - f(x^\ast) \geq \eps$ or otherwise we have $f\pa{x_t} - f\pa{x^\ast} \leq \eps$ by Lemma \ref{lem:three_upper} (1) or $f(x_\theta) - f(x^\ast) \leq \eps$ by direct negation of $f(x_\theta) - f(x^\ast) \geq \eps$, in either case we have found a desired solution. Now for $\omega\pa{\theta} = 4d^\frac{5(p-2)}{2p}\pa{6 + 9LA_t + \frac{L D_R}{f(x_\theta) - f(x^\ast)}}$, we have
    \begin{align*}
        10\omega\pa{\theta} &= 40d^\frac{5(p-2)}{2p}\pa{6 + 9LA_t + \frac{L D_R}{f(x_\theta) - f(x^\ast)}} \\
        &\leq 40d^\frac{5(p-2)}{2p}\pa{6 + \frac{9LR^2}{2\eps} + \frac{2L D_R}{\eps}} \\
        &\leq 320d^\frac{5(p-2)}{2p}\pa{\frac{LD_R}{\eps}}
    \end{align*}
    for $\eps \leq \frac{LD_R}{6}$ in which $D_R = \pa{R + 1458R^2}\pa{20R + 4374R^2}$.
    Therefore, by choosing $\delta = \frac{\eps}{320d^\frac{5(p-2)}{2p}LD_R}$, we have $\abs{\theta - \theta^\ast} \leq \frac{1}{10\omega\pa{\theta}}$, and by Lemma \ref{lem:zeta_range} we confirm that $\zeta\pa{\theta} \in \bracks{\frac{1}{2}, 2}$. And the complexity of binary search to find such $\theta$ is $\log_2\pa{\frac{1}{\delta}} \leq 9+\frac{5(p-2)}{2p}\log_2\pa{d} + \log_2\pa{\frac{LD_R}{\eps}}$.
    
\end{proof}

\subsection{Proof of Supporting Lemmas}

\begin{lemma} \label{lem:psi_upper} $\forall \ t \in [T]$,
    $\psi_t\pa{x} \leq A_t f\pa{x} + \frac{1}{2}\norm{x-x_0}_2^2$.
\end{lemma}
\begin{proof} \
    By definition,
    \begin{align*}
        \psi_t(x) &= \psi_{t-1}(x) + a_t[f(x_t) + \inner{\grad f(x_t)}{x-x_t}] \\
        &\leq \psi_{t-1}(x) + a_tf(x) \\
        &\leq \psi_0(x) + \sum_{i=1}^t a_i f(x) \\
        &= \frac{1}{2}\norm{x-x_0}_2^2 + A_t f\pa{x},
    \end{align*}
    where the first inequality holds by convexity and the second by applying the first recursively.
\end{proof}

\begin{lemma} \label{lem:three_upper} $\forall \ t \in [T]$,
\begin{itemize}
    \item [(1)] $f\pa{x_t} - f\pa{x^\ast} \leq \frac{1}{2A_t} \norm{x_0 - x^\ast}_2^2$,
    \item [(2)] $\norm{v_t - x^\ast}_p \leq \norm{v_t - x^\ast}_2 \leq \norm{x_0 - x^\ast}_2 = R$,
    \item [(3)] $B_t \leq \frac{1}{2} \norm{x_0 - x^\ast}_2^2$.
\end{itemize}
\end{lemma}
\begin{proof} \
Given the definition that $\forall \ t \in [T]$, $\psi_t(x) = \psi_{t-1}(x) + a_t[f(x_t) + \inner{\grad f(x_t)}{x-x_t}]$, we have $\grad \psi_t(x) = \grad \psi_{t-1}(x) +  a_t \grad f(x_t)$. Applying this equality recursively, $\grad \psi_t(x) = \grad \psi_0(x) + \sum_{i=1}^t a_i \grad f(x_i)$. Given that $\psi_0(x) = \frac{1}{2}\norm{x - x_0}_2^2$, we have $\forall \ t \in [T]$, $\grad^2 \psi_t(x) = \grad^2 \psi_0(x) = \bf{I}$ and third-order derivative of $\psi_t(x)$ being zero. As a result, we have for the second-order Taylor expansion of $\psi_t(x)$ at $v_t$ that $\forall \ x$,
        \begin{align}
            \psi_t(x) &= \psi_t(v_t) + \inner{\nabla \psi_t(v_t)}{x - v_t} + \frac{1}{2}\grad^2\psi_t(v_t)\norm{x-v_t}_2^2 \nonumber \\
            &= \psi_t(v_t) + \frac{1}{2}\norm{x-v_t}_2^2 \label{eq:phi_vt}
        \end{align}
        where the first-order term vanishes by the definition $v_t = \argmin\limits_{x\in \R^d} \psi_t(x)$ which indicates that $\nabla \psi_t(v_t) = 0$. Plugging in $x^\ast$, we have by Lemma \ref{lem:recurrence},
        \begin{align*}
            A_t f(x_t) + B_t &\leq \min\limits_{x \in \R^d} \psi_t(x) \\
            &\leq \psi_t(v_t) \\
            &= \psi_t(x^\ast) - \frac{1}{2}\norm{x^\ast-v_t}_2^2 \\
            &\leq A_t f(x^\ast) + \frac{1}{2}\norm{x^\ast-x_0}_2^2 - \frac{1}{2}\norm{x^\ast-v_t}_2^2 
        \end{align*}
        where the equality follows from Eq. \eqref{eq:phi_vt} and the last inequality from Lemma \ref{lem:psi_upper}. Rearranging the terms, we have 
        \begin{align*}
            A_t \bracks{f(x_t) - f(x^\ast)} + B_t + \frac{1}{2}\norm{x^\ast-v_t}_2^2 \leq \frac{1}{2}\norm{x^\ast-x_0}_2^2.
        \end{align*}
        Given that $A_t \geq 0$, $f(x_t) - f(x^\ast) \geq 0$, $B_t \geq 0$, and $\frac{1}{2}\norm{x^\ast-v_t}_2^2 \geq 0$, we have $A_t \bracks{f(x_t) - f(x^\ast)} \leq \frac{1}{2}\norm{x^\ast-x_0}_2^2$ which yields (1), $\frac{1}{2}\norm{x^\ast-v_t}_2^2 \leq \frac{1}{2}\norm{x^\ast-x_0}_2^2$ which yields (2), and $B_t \leq \frac{1}{2}\norm{x^\ast-x_0}_2^2$ which completes the proof.
\end{proof}
   
\begin{lemma} \label{lem:R_bounds} $\forall \ t \in [T]$, for $\norm{x_0 - x^\ast}_2 = R$,
    \begin{itemize}
    \item [(1)] $\norm{x_t - x^\ast}_p \leq \norm{x_0 - x^\ast}_2 + \frac{2916}{t}\norm{x_0 - x^\ast}_2^2 = R + 2916R^2$,
    \item [(2)] $\norm{x_t - v_t}_p \leq 2R + 2916R^2$.
\end{itemize}
\end{lemma} 
\begin{proof}
    \begin{itemize}
        \item [(1)] By definition, $y_t = \frac{A_t}{A_{t+1}}x_t + \frac{a_{t+1}}{A_{t+1}}v_t$. Thus,
        \begin{align*}
            \norm{y_t - x^\ast}_p &= \norm{\frac{A_t}{A_{t+1}}x_t + \frac{a_{t+1}}{A_{t+1}}v_t - \frac{A_t}{A_{t+1}}x^\ast - \frac{a_{t+1}}{A_{t+1}}x^\ast}_p \\
            &\leq \frac{A_t}{A_{t+1}}\norm{x_t - x^\ast}_p + \frac{a_{t+1}}{A_{t+1}}\norm{v_t - x^\ast}_p
        \end{align*}
        Also, by Lemma \ref{lem:progress} Eq. \eqref{eq:gradnorm} we have
        \begin{align*}
            L\norm{x_{t+1} - y_{t}}_p^2 &\leq \inner{\grad f(x_{t+1})}{y_t - x_{t+1}} \\
            &\leq \norm{\grad f(x_{t+1})}_{p^*} \norm{y_t - x_{t+1}}_p,
        \end{align*}
        which indicates that 
        \begin{align}
            \norm{x_{t+1} - y_{t}}_p \leq \frac{1}{L}\norm{\grad f(x_{t+1})}_{p^*}. \label{eq:norm_grad_z}
        \end{align}
        Then we have
        \begin{align*}
            \norm{x_{t+1} - x^\ast}_p &= \norm{x_{t+1} - y_t + y_t - x^\ast}_p \\
            &\leq \norm{x_{t+1} - y_t}_p + \norm{y_t - x^\ast}_p \\
            &\leq \norm{x_{t+1} - y_t}_p + \frac{A_t}{A_{t+1}}\norm{x_t - x^\ast}_p + \frac{a_{t+1}}{A_{t+1}}\norm{v_t - x^\ast}_p \\
            &\leq \frac{1}{L}\norm{\grad f(x_{t+1})}_{p^*} + \frac{A_t}{A_{t+1}}\norm{x_t - x^\ast}_p + \frac{a_{t+1}}{A_{t+1}}\norm{x_0 - x^\ast}_p
        \end{align*}
        where the last inequality follows from Eq. \eqref{eq:norm_grad_z} and Lemma \ref{lem:R_bounds} (1). Applying this inequality recursively,
        \begin{align*}
            \norm{x_{t+1} - x^\ast}_p &\leq \frac{A_0}{A_{t+1}}\norm{x_0 - x^\ast}_p + \frac{\sum_{i=1}^{t+1}a_i}{A_{t+1}}\norm{x_0 - x^\ast}_p + \frac{1}{L} \sum_{i=0}^t\frac{A_{i+1}}{A_{t+1}}\norm{\grad f(x_{i+1})}_{p^*} \\
            &= \norm{x_0 - x^\ast}_p + \frac{1}{L}\sum_{i=0}^t\frac{A_{i+1}^\frac{1}{2}}{A_{t+1}} A_{i+1}^\frac{1}{2}\norm{\grad f(x_{i+1})}_{p^*} \\
            &\leq \norm{x_0 - x^\ast}_p + \frac{1}{L}\pa{\sum_{i=0}^t\frac{A_{i+1}}{A_{t+1}^2}}\pa{\sum_{i=0}^t A_{i+1}\norm{\grad f(x_{i+1})}_{p^*}^2} \\
            &= \norm{x_0 - x^\ast}_p + \pa{18\sum_{i=0}^t\frac{A_{i+1}}{A_{t+1}^2}}B_{t+1} \\
            &\leq \norm{x_0 - x^\ast}_p + \pa{9\sum_{i=0}^t\frac{A_{i+1}}{A_{t+1}^2}}\norm{x_0 - x^\ast}_2^2
        \end{align*}
        where the first equality follows from $A_0 = 0$ and $A_{t+1} = \sum_{i=1}^{t+1}a_i$, for the second inequality we applied Cauchy-Schwarz inequality, and for the last inequality we applied Lemma \ref{lem:three_upper} (3). Furthermore,
        \begin{align*}
            \sum_{i=0}^t\frac{A_{i+1}}{A_{t+1}^2} &= \frac{1}{A_{t+1}}\sum_{i=0}^t\frac{A_{i+1}}{A_{t+1}} \\
            &\leq \frac{1}{A_{t+1}}\sum_{i=0}^t\frac{A_{i+1}}{A_{i+1}} \\
            &= \frac{t+1}{A_{t+1}} \\
            &\leq \frac{t+1}{\pa{\frac{\mathcal{G}_{t+1} (t+1)}{18L^{1/2}}}^2} \\
            &\leq \frac{324}{t+1}
        \end{align*}
        where the first inequality follows from $A_{t+1} \geq A_{i+1}$ for $i \leq t$ since $A_{t+1} = \sum_{i=1}^{t+1}a_i$, and the second inequality from Lemma \ref{lem:growth}. Therefore,
        \begin{align*}
            \norm{x_{t+1} - x^\ast}_p &\leq \norm{x_0 - x^\ast}_p + \frac{2916}{t+1}\norm{x_0 - x^\ast}_2^2.
        \end{align*}
        \item [(2)] $\norm{x_t - v_t}_p = \norm{x_t - x^\ast + x^\ast - v_t}_p \leq \norm{x_t - x^\ast}_p + \norm{v_t - x^\ast}_p \leq 2R + 2916R^2$ applying Lemma \ref{lem:three_upper} (2) and Lemma \ref{lem:R_bounds} (1).
    \end{itemize}
\end{proof}

\begin{lemma}[Proposition 3 in \citep{balles2020geometry} ] \label{lem:hessian} $f\colon \R^d \rightarrow \R$ is $L$-smooth in the norm $\norm{\cdot}_p$ if and only if $\forall \ x \in \R^d$, $\norm{\grad^2 f(x)}_p \leq L$.
\end{lemma}

\begin{lemma} \label{lem:lp_hessian} For $z \in \R^d$ and $s\pa{z} = \bracks{\abs{z[1]}^{p-2}z[1], \cdots, \abs{z[d]}^{p-2}z[d]}^\top$,
    \begin{itemize}
        \item [(1)] $\grad^2_z\norm{z}_p^2 = 2(p-1)\norm{z}_p^{2-p}\mathrm{Diag}\pa{\abs{z[1]}^{p-2}, \cdots, \abs{z[d]}^{p-2}} + 2(2-p)\norm{z}_p^{2(1-p)}s\pa{z} s\pa{z}^\top$,
        \item [(2)] $\grad^2_z\norm{z}_p^2 \succeq 2\norm{z}_p^{2(1-p)}s\pa{z} s\pa{z}^\top$,
        \item [(3)] $\norm{\grad^2_z\norm{z}_p^2}_p \geq \frac{2}{d^\frac{p-2}{2}}$.
    \end{itemize}
\end{lemma}
\begin{proof}
    \begin{itemize}
        \item [(1)] Given $\norm{z}_p = \pa{\sum_{i=1}^d\abs{z[i]}^p}^\frac{1}{p}$, we have
        \begin{align*}
            \grad_z\norm{z}_p^2 = 2\norm{z}_p\begin{bmatrix}
           \frac{1}{p}\norm{z}_p^{1-p}\cdot p \abs{z[1]}^{p-2}z[1] \\
           \vdots \\
           \frac{1}{p}\norm{z}_p^{1-p}\cdot p \abs{z[d]}^{p-2}z[d]
         \end{bmatrix}
        = 2\norm{z}_p^{2-p}\begin{bmatrix}
           \abs{z[1]}^{p-2}z[1] \\
           \vdots \\
           \abs{z[d]}^{p-2}z[d]
         \end{bmatrix}
        \end{align*}
        Therefore, $\forall \ i, j \in [d]$,
        \begin{align*}
            \frac{\partial^2}{\partial z[i]\partial z[j]} = \begin{cases}
            2(2-p)\norm{z}_p^{2(1-p)}\abs{z[i]}^{2(p-1)} + 2(p-1)\norm{z}_p^{2-p}\abs{z[i]}^{p-2} & \text{if } i=j \\
            2(2-p)\norm{z}_p^{2(1-p)}\abs{z[i]}^{p-2} z[i] \abs{z[j]}^{p-2} z[j] & \text{if } i \neq j
            \end{cases},
        \end{align*}
        which yields $$\grad^2_z\norm{z}_p^2 = 2(2-p)\norm{z}_p^{2(1-p)}s\pa{z} s\pa{z}^\top + 2(p-1)\norm{z}_p^{2-p}\mathrm{Diag}\pa{\abs{z[1]}^{p-2}, \cdots, \abs{z[d]}^{p-2}}.$$
        \item [(2)] By the result of (1),
        \begin{align*}
            \grad^2_z\norm{z}_p^2 &= 2(p-1)\norm{z}_p^{2-p}\mathrm{Diag}\pa{\abs{z[1]}^{p-2}, \cdots, \abs{z[d]}^{p-2}} \\
            &\quad - 2(p-1)\norm{z}_p^{2(1-p)}s\pa{z} s\pa{z}^\top + 2\norm{z}_p^{2(1-p)}s\pa{z} s\pa{z}^\top \\
            &= 2(p-1)\norm{z}_p^{2(1-p)}\pa{\underbrace{\norm{z}_p^p\mathrm{Diag}\pa{\abs{z[1]}^{p-2}, \cdots, \abs{z[d]}^{p-2}} - s\pa{z} s\pa{z}^\top}_M} \\
            &\quad + 2\norm{z}_p^{2(1-p)}s\pa{z} s\pa{z}^\top.
        \end{align*}
        We now show that matrix $M$ is positive semidefinite. $\forall \ v \in \R^d$, 
        \begin{align*}
            v^\top M v &= \norm{z}_p^p v^\top\mathrm{Diag}\pa{\abs{z[1]}^{p-2}, \cdots, \abs{z[d]}^{p-2}}v - v^\top s\pa{z} s\pa{z}^\top v\\
            &= \norm{z}_p^p \sum_{i=1}^d v_i^2 \abs{z[i]}^{p-2} - \pa{v^\top s\pa{z}}^2 \\
            &= \pa{\sum_{i=1}^d \abs{z[i]}^p} \pa{\sum_{i=1}^d v_i^2 \abs{z[i]}^{p-2}} - \pa{\sum_{i=1}^d v_i \abs{z[i]}^\frac{p-2}{2} \cdot \abs{z[i]}^\frac{p-2}{2}z[i]}^2 \\
            &\geq \pa{\sum_{i=1}^d \abs{z[i]}^p} \pa{\sum_{i=1}^d v_i^2 \abs{z[i]}^{p-2}} - \pa{\sum_{i=1}^d v_i^2 \abs{z[i]}^{p-2}}\pa{\sum_{i=1}^d\abs{z[i]}^{p-2}z[i]^2} \\
            &= 0
        \end{align*}
        where we applied the Cauchy-Schwarz inequality. Therefore, we have $$\grad^2_z\norm{z}_p^2 - 2\norm{z}_p^{2(1-p)}s\pa{z} s\pa{z}^\top = 2(p-1)\norm{z}_p^{2(1-p)}M \succeq \bf{0},$$ which completes the proof.
        \item [(3)] By the result of (2),
        \begin{align*}
            \norm{\grad^2_z\norm{z}_p^2}_p \geq 2\norm{z}_p^{2(1-p)}\norm{s\pa{z} s\pa{z}^\top}_p.
        \end{align*}
        By definition of induced matrix $\ell_p$-norm, 
        \begin{align*}
            \norm{s\pa{z} s\pa{z}^\top}_p &= \sup_{v\colon \norm{v}_p=1} \norm{s\pa{z} s\pa{z}^\top v}_p \\
            &= \norm{s\pa{z}}_p \sup_{v\colon \norm{v}_p=1} \abs{s\pa{z}^\top v} \\
            &= \norm{s\pa{z}}_p \norm{s\pa{z}}_{p^*} \\
            &\geq \norm{s\pa{z}}_2^2 \\
            &= \sum_{i=1}^d \abs{z[i]}^{2(p-1)} \\
            &= \norm{z}_{2(p-1)}^{2(p-1)}
        \end{align*}
        where the second equality uses the fact that $s\pa{z}^\top v$ is a scalar, the third equality follows from the definition of dual vector norm, and the inequality follows the Cauchy-Schwarz inequality. Finally,
        \begin{align*}
            \norm{\grad^2_z\norm{z}_p^2}_p \geq \frac{2\norm{z}_{2(p-1)}^{2(p-1)}}{\norm{z}_p^{2(p-1)}} \geq 2\pa{\frac{\norm{z}_{2(p-1)}}{d^{\frac{1}{p}-\frac{1}{2(p-1)}}\norm{z}_{2(p-1)}}}^{2(p-1)} = \frac{2}{d^\frac{p-2}{2}}.
        \end{align*}
    \end{itemize}
\end{proof}

\begin{lemma} \label{lem:d_theta}
    $\norm{\frac{d}{d\theta}x_\theta}_p = \norm{\frac{d}{d\theta}z_\theta}_p \leq d^\frac{p-2}{p}\pa{R + 1458R^2}$ for $z_\theta = x_\theta - y_\theta$.
\end{lemma}
\begin{proof} \ 
Denote $F(x_\theta, y_\theta) = f(y_\theta) + \inner{\grad f(y_\theta)}{x_\theta - y_\theta} + L\norm{x_\theta - y_\theta}_p^2$. By the definition that $x_\theta = \argmin\limits_{x \in \R^d} F(x, y_\theta)$ and first optimality condition, we know that $$\grad_x F(x_\theta, y_\theta) = \grad f(y_\theta) + \grad_x \pa{L\norm{x_\theta - y_\theta}_p^2} = 0.$$ Taking the gradient with respect to $\theta$ on both sides yields
\begin{align}
    \grad^2_x F(x_\theta, y_\theta) \frac{d}{d\theta}x_\theta + \grad_y\grad_x F(x_\theta, y_\theta) \frac{d}{d\theta}y_\theta = 0, \label{eq:grad_optimality}
\end{align}
in which $\grad^2_x F(x_\theta, y_\theta) = \grad^2_x \pa{L\norm{x_\theta - y_\theta}_p^2}$, $\frac{d}{d\theta}y_\theta = x_t - v_t$, and \begin{align*}
    \grad_y\grad_x F(x_\theta, y_\theta) &= \grad^2 f(y_\theta) + \grad_y \grad_x \pa{L\norm{x_\theta - y_\theta}_p^2} \\
    &= \grad^2 f(y_\theta)
\end{align*}
since $\grad_y \grad_x \pa{L\norm{x_\theta - y_\theta}_p^2} = 0$ by calculation or by symmetry. Therefore, from Eq. \eqref{eq:grad_optimality},
\begin{align*}
    \frac{d}{d\theta}x_\theta = \pa{\grad^2_x \pa{L\norm{x_\theta - y_\theta}_p^2}}^{-1}\grad^2 f(y_\theta)\pa{v_t - x_t}.
\end{align*}
Taking the $\ell_p$-norm on both sides
    \begin{align*}
        \norm{\frac{d}{d\theta}x_\theta}_p &\leq \frac{\norm{\grad^2 f(y_\theta)} \norm{v_t-x_t}_p}{L\norm{\grad^2_x \pa{\norm{x_\theta - y_\theta}_p^2}}_p} \\
        &\leq \frac{d^\frac{p-2}{p}}{2} \norm{x_t - v_t}_2 \\
        &\leq d^\frac{p-2}{p}\pa{R + 1458R^2}
    \end{align*}
    where the second inequality follows from Lemma \ref{lem:hessian} and \ref{lem:lp_hessian}, and the last inequality follows from Lemma \ref{lem:R_bounds}. By letting $F(z_\theta, y_\theta) = f(y_\theta) + \inner{\grad f(y_\theta)}{z_\theta} + L\norm{z_\theta}_p^2$ and following the same argument, we can show the result also holds for $ \norm{\frac{d}{d\theta}z_\theta}_p$.
\end{proof}

\begin{lemma} \label{lem:z_bound}
    $\norm{x_\theta - y_\theta}_p \leq 18R + d^\frac{p-2}{p}\pa{R + 1458R^2}$.
\end{lemma}
\begin{proof} \ 
By $\ell_p$-smoothness we have $f(x_\theta) \leq f(y_\theta) + \inner{\grad f(y_\theta)}{x_\theta - y_\theta} + \frac{L}{2}\norm{x_\theta - y_\theta}_p$. Then for
\begin{align*}
    F(x_\theta, y_\theta) &= f(y_\theta) + \inner{\grad f(y_\theta)}{x_\theta - y_\theta} + L\norm{x_\theta - y_\theta}_p^2 \\
    &\geq f(x_\theta) - \frac{L}{2}\norm{x_\theta - y_\theta}_p^2 + L\norm{x_\theta - y_\theta}_p^2 \\
    &= f(x_\theta) + \frac{L}{2}\norm{x_\theta - y_\theta}_p^2.
\end{align*}
As a result, rearranging the terms we get
\begin{align*}
    \norm{x_\theta - y_\theta}_p^2 &\leq \frac{2}{L} \pa{F(x_\theta, y_\theta) - f(x_\theta)} \\
    &\leq \frac{2}{L} \pa{F(x_\theta, y_\theta) - f(x^\ast)} \\
    &\leq \frac{2}{L} \pa{F(y_\theta, y_\theta) - f(x^\ast)} \\
    &=  \frac{2}{L} \pa{f(y_\theta) - f(x^\ast)}
\end{align*}
where the last inequality follows from that $x_\theta = \argmin\limits_{x \in \R^d} F(x, y_\theta)$. When $\theta = 1$, we have $y_{\theta=1} = x_t$. Then we have
\begin{align*}
    \norm{x_{\theta=1} - y_{\theta=1}}_p^2 &\leq  \frac{2}{L} \pa{f(x_t) - f(x^\ast)} \\
    &\leq \frac{1}{LA_t} \norm{x_0 - x^\ast}_2^2 \\
\end{align*}
where the second inequality follows from Lemma \ref{lem:three_upper}. Finally, for $\norm{\frac{d}{d\theta}\pa{x_\theta - y_\theta}}_p = \norm{\frac{d}{d\theta}z_\theta}_p \leq d^\frac{p-2}{p}\pa{R + 1458R^2}$, by Taylor's inequality,
    \begin{align*}
        \norm{x_\theta - y_\theta}_2 &\leq \norm{x_{\theta=1} - y_{\theta=1}}_p + d^\frac{p-2}{p}\pa{R + 1458R^2} \abs{\theta - 1} \\
        &\leq \frac{\norm{x_0 - x^\ast}_2}{\sqrt{L}\frac{\mathcal{G}_t t}{18L^{1/2}}} + d^\frac{p-2}{p}\pa{R + 1458R^2} \\
        &\leq 18R + d^\frac{p-2}{p}\pa{R + 1458R^2}
    \end{align*}
    where we applied Lemma \ref{lem:R_bounds}, \ref{lem:d_theta}, and \ref{lem:growth}.  
\end{proof}

\begin{lemma} \label{lem:grad_lower}
$\norm{\grad f\pa{x_\theta}}_{p^*} \geq \frac{f(x_\theta) - f(x^\ast)}{\pa{19+d^\frac{p-2}{p}}R + \pa{2916 + 1458d^\frac{p-2}{p}}R^2}$.
\end{lemma}
\begin{proof} \
    Starting with the definition of convexity and applying Cauchy-Schwarz inequality,
\begin{align*}
    f(x_\theta) &\leq f(x^\ast) + \inner{\grad f\pa{x_\theta}}{x_\theta - x^\ast} \\
    &\leq f(x^\ast) + \norm{\grad f\pa{x_\theta}}_{p^*} \norm{x_\theta - x^\ast}_p \\
    &\leq f(x^\ast) + \norm{\grad f\pa{x_\theta}}_{p^*} \norm{x_\theta - y_\theta + y_\theta - x^\ast}_p \\
    &\leq f(x^\ast) + \norm{\grad f\pa{x_\theta}}_{p^*} \pa{\norm{x_\theta - y_\theta}_p + \norm{y_\theta - x^\ast}_p} \\
    &\leq f(x^\ast) + \norm{\grad f\pa{x_\theta}}_{p^*} \pa{\norm{x_\theta - y_\theta}_p + \norm{\theta x_t + (1-\theta)v_t - x^\ast}_p} \\
    &\leq f(x^\ast) + \norm{\grad f\pa{x_\theta}}_{p^*} \pa{\norm{x_\theta - y_\theta}_p + \theta\norm{x_t - x^\ast}_p + (1-\theta)\norm{v_t - x^\ast}_p}
\end{align*}
We can bound the three terms in the parenthesis by Lemma \ref{lem:z_bound}, \ref{lem:three_upper}, and \ref{lem:R_bounds} as follows:
\begin{align*}
    f(x_\theta) &\leq f(x^\ast) + \norm{\grad f\pa{x_\theta}}_{p^*} \pa{18R + d^\frac{p-2}{p}\pa{R + 1458R^2} + \theta\pa{R + 2916R^2} + (1-\theta)R} \\
    &\leq  f(x^\ast) + \norm{\grad f\pa{x_\theta}}_{p^*} \pa{\pa{19+d^\frac{p-2}{p}}R + \pa{2916 + 1458d^\frac{p-2}{p}}R^2} 
\end{align*}
Rearranging the terms we have
\begin{align*}
    \norm{\grad f\pa{x_\theta}}_{p^*} \geq \frac{f(x_\theta) - f(x^\ast)}{\pa{19+d^\frac{p-2}{p}}R + \pa{2916 + 1458d^\frac{p-2}{p}}R^2}.
\end{align*}
\end{proof}

\begin{lemma} \label{lem:zeta_range}
    For $\zeta \colon [0,1] \rightarrow \R_+$ that satisfies $\forall \ \theta \in [0,1]$, $\abs{\frac{d}{d\theta}\log\pa{\zeta\pa{\theta}}} \leq \omega \pa{1 + \frac{1}{\zeta\pa{\theta}} + \zeta\pa{\theta}}$, and $\exists \ \theta^\ast \in [0,1]$ such that $\zeta\pa{\theta^\ast} = \frac{5}{4}$, if it holds for $\tilde{\theta}$ such that  $\abs{\tilde{\theta} - \theta^\ast} \leq \frac{1}{10\omega}$ for some $\omega \geq 0$, then one has for such $\tilde{\theta}$ that
    \begin{align*}
        \zeta\pa{\tilde{\theta}} \in \bracks{\frac{1}{2}, 2}.
    \end{align*}
\end{lemma}
\begin{proof} \
    $\forall \ \theta$ such that $\abs{\theta - \theta^\ast} \leq \gamma$, we denote the range of $\zeta\pa{\theta} \in \bracks{\alpha, \beta}.$ Then we have
    \begin{align*}
        \abs{\frac{d}{d\theta}\zeta\pa{\theta}} &= \abs{\zeta\pa{\theta} \frac{1}{\zeta\pa{\theta}} \frac{d}{d\theta}\zeta\pa{\theta}}\\
        &= \abs{\zeta\pa{\theta} \frac{d}{d\theta}\log\pa{\zeta\pa{\theta}}}\\
        &\leq \omega \pa{\zeta\pa{\theta} + 1 + \zeta\pa{\theta}^2} \\
        &\leq \pa{1+\beta+\beta^2}\omega 
    \end{align*}
    By the mean value inequality, 
    \begin{align*}
        \abs{\zeta\pa{\theta} - \zeta\pa{\theta^\ast}} \leq \pa{1+\beta+\beta^2}\omega  \abs{\theta - \theta^\ast} \leq \pa{1+\beta+\beta^2}\omega \gamma.
    \end{align*}
    Let $\gamma = \frac{3}{4\pa{1+\beta+\beta^2}\omega}$, we have $\abs{\zeta\pa{\theta} - \zeta\pa{\theta^\ast}} \leq \frac{3}{4}$, which implies that if $$\abs{\theta - \theta^\ast} \leq \frac{3}{4\pa{1+\beta+\beta^2}\omega}$$ holds, then one has $\frac{1}{2} = \zeta\pa{\theta^\ast} - \frac{3}{4} \leq \zeta\pa{\theta} \leq \zeta\pa{\theta^\ast} + \frac{3}{4} = 2,$
    i.e., $\zeta\pa{\theta} \in \bracks{\frac{1}{2}, 2}$. Therefore, we must have $\bracks{\alpha, \beta} \subset \bracks{\frac{1}{2}, 2}$, i.e., $\beta \leq 2$.
    As a result, we verify for $\tilde{\theta}$, 
    $$\abs{\tilde{\theta} - \theta^\ast} \leq \frac{1}{10\omega} \leq \frac{3}{4\pa{1+2+2^2}\omega} \leq \frac{3}{4\pa{1+\beta+\beta^2}\omega},$$
    and therefore we have $\zeta\pa{\tilde{\theta}} \in \bracks{\frac{1}{2}, 2}.$
\end{proof}

\section{Empirical Validation}\label{app:experiments}
We provide in this section experimental evidence to validate (and complement) our theoretical contributions. In order to do so, we consider the function $\text{LogSumExp}(v)$ (related to the softmax function, e.g., \citep{kelner2014almost, bullins2020highly}), which is defined as
\begin{equation*}
    \text{LogSumExp}(v) \defeq \log\pa{\sum\limits_{i=1}^n e^{v[i]}}\quad.
\end{equation*} 
Next, we consider $A \in \R^{n \times d}$ such that each $A_{i,j}$ is drawn from a Bernouilli with $p=0.8$, and $b \in \R^n$ such that each $b_i$ is drawn from a normal distribution with mean 0 and variance 1. Following from this, the functions we aim to minimize are
\begin{equation*}
f(x) = \text{LogSumExp}(Ax-b) + \frac{\mu}{2}\norm{x}^2,
\end{equation*}
for $\mu \in \{0, 1e-6, 1e-4, 1e-2\}$. The experiments were run on a MacBook Pro laptop computer.

\begin{figure}
\includegraphics[width=\linewidth]{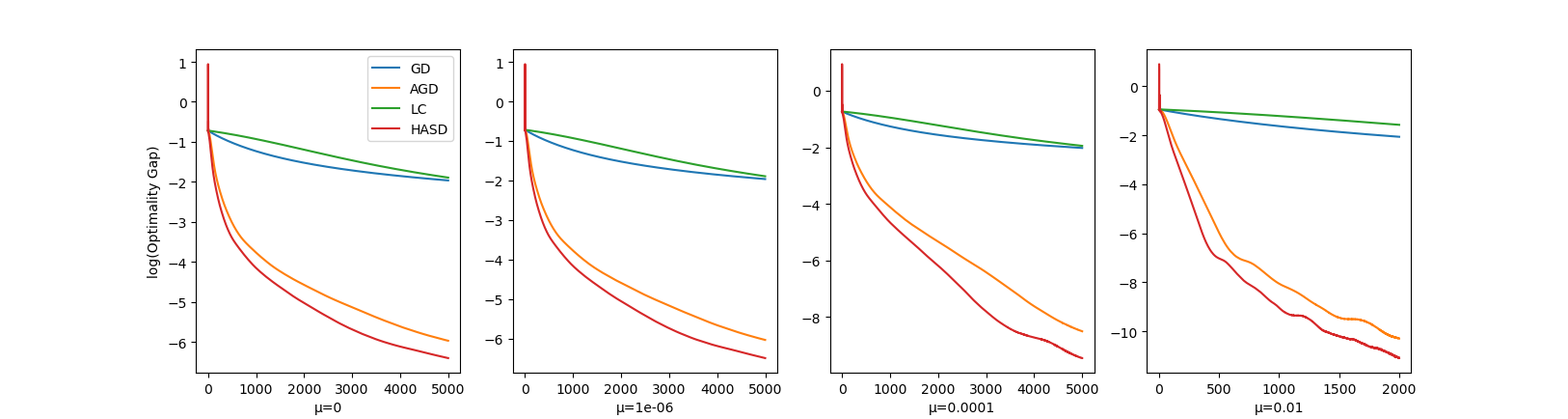}
\caption{Comparison of HASD (our method) with Gradient Descent (GD), Accelerated Gradient Descent (AGD), and LC w.r.t.~$\norm{\cdot}_\infty$ (LC), across different choices of $\mu \in \{0, 1e-6, 1e-4, 1e-2\}$. All plots have been parameter tuned (in terms of stepsize) based on the set of possible stepsizes provided. The x-axis of each plot represents the number of iterations, and the y-axis represents $\log(\text{Optimality Gap})$. Notably, the addition of the implicit coupling term leads to significant improvements over the (non-implicitly-coupled) LC method, and our algorithm performs in a manner comparable to (or slightly better than) AGD.}\label{fig:experiments}
\centering
\end{figure}

We compare the methods \textbf{Gradient Descent} (GD), \textbf{Accelerated Gradient Descent} (AGD), \textbf{Linear Coupling} w.r.t.~$\norm{\cdot}_\infty$ (LC), and our methods \textbf{Hyper-Accelerated Steepest Descent} (HASD) w.r.t.~$\norm{\cdot}_\infty$. We tune the stepsize parameter over the set $\{1e-10, 2e-10, 5e-10, 1e-9, 2e-9, 5e-9, 1e-8, 2e-8, 5e-8, 1e-7, 2e-7, 5e-7, 1e-6, 2e-6, 5e-6, 1e-5,2e-5,5e-5,1e-4,2e-4,5e-4,1e-3,2e-3,5e-3,1e-2,2e-2,5e-2,1e-1,2e-1,5e-1,1\}$, and the results may be found in Figure~\ref{fig:experiments}. Notably, our algorithm performs slightly better than AGD, and significantly better than its (non-implicitly-coupled) counterpart LC.

\end{document}